\documentclass[11pt,a4paper]{article}
\usepackage[utf8]{inputenc}
\usepackage[english]{babel}
\usepackage{amsmath}
\usepackage{amsfonts}
\usepackage{amssymb}
\usepackage{amsthm}
\usepackage{graphicx}
\usepackage[framemethod=TikZ]{mdframed}
\usepackage[margin=1in]{geometry}
\usepackage{soul}
\usepackage{hyperref}
\usepackage{authblk}
\usepackage{mathtools}
\usepackage{esint}

\usepackage{xcolor}

\newcommand{\ep}{\epsilon}

\newtheorem{theorem}{Theorem}[section]
\newtheorem{lemma}[theorem]{Lemma}
\newtheorem{claim}[theorem]{Claim}
\newtheorem{proposition}[theorem]{Proposition}
\newtheorem{corollary}[theorem]{Corollary}

\theoremstyle{definition}
\newtheorem{definition}[theorem]{Definition}

\theoremstyle{remark}
\newtheorem{remark}[theorem]{Remark}

\numberwithin{equation}{section}

\begin{document}
\selectlanguage{english}

\title{A Hamilton-Jacobi Approach to Evolution of Dispersal}

\author{King-Yeung Lam}
\affil{Department of Mathematics, Ohio State University, Columbus, OH,
USA\\
\texttt{lam.184@math.ohio-state.edu}}

\author{Yuan Lou}
\affil{School of Mathematical Sciences, CMA-Shanghai, Shanghai Jiao Tong University, Shanghai
200240,
China\\
\texttt{yuanlou@sjtu.edu.cn}}

\author{Beno\^\i t Perthame} 
\affil{Sorbonne Universit\'{e}, CNRS, Université de Paris, Inria, Laboratoire Jacques-Louis Lions, F-75005 Paris, France\\
\texttt{benoit.perthame@sorbonne-universite.fr}}
\date{}
\maketitle
\thispagestyle{empty}

\affil[$\star$]{}

%
%
\begin{abstract}
The evolution of dispersal is a classical question in evolutionary biology, and it has been studied in a wide range of mathematical models. A selection-mutation model,  in which the population is structured by space and a phenotypic trait, with the trait acting directly on the dispersal (diffusion) rate, was formulated by Perthame and Souganidis [Math. Model. Nat. Phenom. 11 (2016), 154–166] to study the evolution of random dispersal towards the evolutionarily stable strategy. For the rare mutation limit, it was shown that the equilibrium population concentrates on a single trait associated to the smallest dispersal rate.  In this paper, we consider the corresponding evolution equation and characterize  the asymptotic behaviors of the time-dependent solutions in the rare mutation limit, under mild convexity assumptions on the underlying Hamiltonian function. 
\end{abstract}
%

\section{Introduction}\label{sec:1}

Recently, various mathematical models for evolutionary biology have been developed with the theories of adaptive dynamics, competition/selection equations, deterministic or stochastic. Here, we are interested in the specific case of the evolution of dispersal in a bounded domain and explain, in a continuous setting, the selection of the slowest. We introduce the framework in a general setting without space before describing the full model.

\subsection{Continuous modeling without space}

A convenient modeling background, based on nonlocal Lotka-Volterra parabolic equations, is considered in \cite{Barles2009,Lorz2011,Perthame2008}. It is expressed under the form
\begin{equation}\label{eq:LV}
\begin{cases}
\ep\partial_t \widetilde{n}_\ep = \ep^2 
\Delta_z \widetilde{n}_\ep + \widetilde{n}_\ep R(z, \widetilde\rho_\ep(t)) &\text{ for }z \in \mathbb{R}^N, \, t>0,\\
\widetilde\rho_\ep(t) = \int_{\mathbb{R}^N} \widetilde{n}_\ep(z,t)\,dz &\text{ for }t >0,\\
\widetilde{n}_\ep(z,0) = \widetilde{n}_{\ep,in}(z) &\text{ for }z \in \mathbb{R}^N.
\end{cases}
\end{equation}
The model \eqref{eq:LV}
describes the dynamics of a population with density  $\widetilde{n}_\ep(z,t)$,
which is structured by a physiological trait $z\in \mathbb{R}^N$. The population dynamics of \eqref{eq:LV}
is driven by both mutation and selection. The mutation process is modeled 
by the Laplacian $\Delta_z$, and the selection 
is expressed by the dependence of the (sign-changing) growth rate of an individual,
given by $R(z,\widetilde\rho_\ep(t))$, on the trait $z$ and the limiting total population $\widetilde\rho_\ep(t)$. We refer to \cite{Champagnat2006} for a derivation of this type of equations from individual based stochastic models. A striking property of
\eqref{eq:LV} is that, as $\ep \to 0$, the generic solution of this equation concentrates as a moving Dirac mass, i.e.  
\begin{equation}
\label{concentration} 
 \widetilde{n}_{\ep}(t) \approx \widetilde\rho(t)\delta(z-\bar{z}(t)) \quad \text{ in distribution},
\end{equation}
where $\widetilde\rho(t) = \lim\limits_{\ep \to 0} \widetilde\rho_\ep(t)$ and the trajectory $\bar{z}(t)$ can be interpreted biologically as the dominant trait at the time $t$. Furthermore, the trajectory $\bar{z}(t)$ can be inferred form the total population $\widetilde\rho(t)$ via the relation 
\begin{equation}\label{eq:RZR}
R(\bar{z}(t),\widetilde\rho(t))=0\quad  a.e.
\end{equation}
and the rare mutation limit of the WKB-ansatz $\widetilde{u}_\ep(z,t) := -\ep \log \widetilde{n}_\ep(z,t)$ and the total population $\widetilde\rho(t)$ satisfy, in the viscosity sense, the following Hamilton-Jacobi equation with a constraint.
\begin{equation}\label{eq:urho}
    \begin{cases}
\partial_t \widetilde{u} + |\partial_z \widetilde{u}|^2 + R(z, \widetilde\rho(t))=0 &\text{ for }z \in \mathbb{R}^N, \, t>0,\\
\inf_z \widetilde{u}(\cdot, t)=u(\bar{z}(t),t)=0 & \text{ for } t>0,\\
\widetilde{u}(z,0) = \widetilde{u}_{in}(z) &\text{ for }z \in \mathbb{R}^N.
\end{cases}
\end{equation}
The key to understanding the evolutionary dynamics $\bar{z}(t)$ is thus contained in the question of uniqueness of solution $(\widetilde{u},\widetilde\rho)$ of \eqref{eq:urho}, which was investigated in \cite{Calvez2018,Mirrahimi2016, Perthame2008}.

The model \eqref{eq:LV} enables a rigorous derivation of the so-called {\it canonical equation}, which has been formally proposed in the framework of adaptive dynamics to describe the trait evolution. See also \cite{Desvillettes2008,Jabin2011,Lorz2011,Perthame2008}. In \cite{Calvez2018}, it is further proved that $\widetilde\rho(t)$ is a nondecreasing function. Since the dominant trait satisfies \eqref{eq:RZR}, this leads to the conclusion that evolution favors the traits that can retain the proliferative advantage when the level of the total population is high.

\subsection{Discrete modeling with spatial structure}

In many biological situations, however, the proliferative advantage is not so directly linked to the physiological trait in question. This is the case, in particular, in the study of evolution of dispersal where individuals adopt different dispersal strategies without an apparent proliferative advantage. 
An 
interesting question in this direction is the selection of random dispersal rate \cite{Hastings1983}. For $K$ interacting populations
with densities
$n_i(x,t)$, $1\le i\le K$,  competing for a common resource $m(x)$ and differing only in their dispersal rates $\alpha_i$, their population dynamics can be described by the following competition system introduced in \cite{DHMP}: for $1 \leq i \leq K$,
\begin{equation}\label{eq:LVV}
\partial_t n_i= \alpha_i \Delta_x n_i + n_i\left(m(x) - \sum_{j=1}^K n_j\right) + \sum_{j=1}^K M_{ij} n_j\quad \text{ for }x \in D,\, t>0.\, 
\end{equation} 
In this context, these $K$
species are (i) subject to a mutation with switching rate $M_{ij}$ from
phenotype $j$ to $i$; (ii) moving randomly with diffusion rates $\alpha_i$ and (iii) competing for a common resource. The proliferative advantage of an individual is influenced by the way the individual moves and utilizes resources. For the two-species case without mutation ($M_{ij} \equiv 0$), a well-known result due to Dockery et al. \cite{DHMP} says that if $m(x)$ is non-constant, and the populations are subject to the no-flux boundary conditions, then the species with the smaller dispersal rate always exclude the faster competitor. 
The corresponding question for 
system \eqref{eq:LVV}
consisting of three or more species remains open; see \cite{Cantrell2021evolution} for recent progress. 

The system \eqref{eq:LVV} describes the competition among finitely many phenotypes with different dispersal rates, denoted by $\alpha_i$. In the next section, we will introduce a mathematical model with infinitely many phenotypes, 
parameterized by a one-dimensional trait variable $z$. Each $z$ corresponds to a phenotype with dispersal rate $\alpha(z)$,
which is a 
continuous function of $z$.

\subsection{The continuous model for evolution of dispersal}
In this paper, we consider a nonlocal reaction-diffusion equation, introduced in \cite{Perthame2015}, which can be viewed as the 
extension of the discrete trait model \eqref{eq:LVV}
to the continuous trait setting.
The main modeling assumptions are:
\\
(i) the dispersal rate of an individual is a positive function of its phenotypic trait $z$, that we denote by $\alpha(z)$;
\\
(ii) a Fisher-type Lotka-Volterra growth 
rate with  spatially heterogeneous carrying capacity $m(x)$ and limitation by the total population at the same spatial location; and
\\
(iii) rare mutation acting on the phenotypic trait variable and modeled by a diffusion with covariance $\sqrt{2}\ep.$
\\


More precisely, we study the asymptotic behavior, as $\epsilon \to 0$, of the density function $n_\ep = n_\ep(x,z,t): D \times I \times [0,T] \to [0,\infty)$, 
of the nonlocal,  nonlinear problem
\begin{equation}\label{eq:1.1}
\left\{
\begin{array}{ll}
\epsilon \partial_t n_\epsilon = \alpha(z) \Delta_x n_{\epsilon}  + n_\epsilon(m(x) - \rho_\epsilon(x,t) )  + \epsilon^2 \partial^2_z n_{\epsilon}&\text{ for  }x \in D,\, z \in I, \, t>0,\\
\partial_\nu n_\epsilon = 0 &\text{ for  } x \in \partial D,\, z \in I,\,  t>0,\\
\partial_z n_\epsilon = 0 &\text{ for }x \in D,\, z\in \partial I,\, t>0,\\
n_\epsilon(x,z,0) = n_{\epsilon,0}(x,z) &\text{ for }x \in D,\, z \in I,
\end{array}
\right.
\end{equation}
\begin{equation}\label{eq:rhoep}
\rho_\epsilon(x,t):= \int_{I} n_\epsilon(x,z,t)\,dz, \qquad x \in D,\, t>0.
\end{equation}
Here  $D$ represents a  bounded spatial domain in $\mathbb{R}^N$ with smooth boundary $\partial D$; $\nu$ denotes the outward unit normal vector on $\partial{D}$; $\Delta_x=\sum_{i=1}^N \partial^2_{x_i}$ is the Laplace operator in $x$ variables;  $I = (a,b) \subset \mathbb{R}$ is a bounded open interval with $a<b$; $\rho_\ep(x,t)$ is the population density of the entire species at location $x$ and time $t$. Without loss of generality we assume $|D|=1$ and $|I|=1$, i.e., $D$ and $I$ are both of unit Lebesgue measure. 
The interest of this problem is that the fitness is controlled by a combination between the dispersal $\alpha(z)$ and the competition for resources expressed by $m(x)- \rho_\ep(x,t)$. It is important that the resource $m(x)$ satisfies
\begin{description}
\item[(M)] $m \in C^\beta(\bar D)$ for
some $\beta\in (0, 1)$,
and it is a nonconstant, positive  function.
\end{description}
On the one hand, if {\bf(M)} does not hold, i.e. $m(x)$ is a constant and $\alpha(z) = z$, then it can be shown that the constant steady state solution is globally asymptotically stable among all nonnegative nontrivial solutions \cite{Lam2017b}. On the other hand, under the assumption {\bf(M)} and some mild assumptions, it can be shown that, for all $\epsilon$ sufficiently small, \eqref{eq:1.1} has a positive steady state solution $\widetilde{n}_\epsilon(x,z)$ concentrated at the trait where $\alpha$ attains the minimum value \cite{Lam2017a,Perthame2015}. Under suitable assumptions, the positive steady state $\widetilde{n}_\epsilon(x,z)$ is shown to be unique and locally asymptotically stable \cite{Lam2017b}. We will enforce hypothesis {\bf(M)} throughout the rest of the paper.

Our goal is to show that, for the evolution of dispersal, the same concentration effect holds around the  fittest trait $\bar z (t)$ as expressed in~\eqref{concentration}.

\subsection{Assumptions and Main Results}

\begin{definition} [Invasion exponent]\label{def:1}
\item 
For each $z_1, z_2 \in I$, let $\lambda(z_1, z_2)$ be the 
smallest eigenvalue of the linear 
equation
\begin{equation}\label{eq:invasionexp}
\left\{
\begin{array}{ll}
\alpha(z_1)\Delta_x \phi + (m-\theta_{z_2})\phi + \lambda \phi=0 &\text{ for }x \in D,\\
\partial_\nu \phi = 0&\text{ for }x \in \partial D,
\end{array}
\right.
\end{equation}
where, for each $z \in I$,  $\theta_{z}(x)$ is the unique positive solution of 
\begin{equation}\label{eq:theta}
\left\{
\begin{array}{ll}
\alpha(z)\Delta_x \theta_{z} + (m-\theta_{z})\theta_z =0 &\text{ for }x \in D,\\
\partial_\nu \theta_z = 0&\text{ for }x \in \partial D.
\end{array}
\right.
\end{equation}
\end{definition}

Notice that for $z_1=z_2=:z$, we have $\lambda(z,z)=0$, which corresponds to $\phi =\theta_{z}$. 
Throughout the paper we impose the following 
assumptions:

\begin{description}

\item[(H1)] $\alpha: I \to \mathbb{R}_+$ is smooth and   is chosen such that $0 < \inf_{I} \alpha \leq \sup_{I} \alpha< +\infty$ and
$$
(\exists K_\lambda>0)\quad 2K_\lambda \leq  \partial^2_{z_1} \lambda(z_1,z_2) \leq \frac{2}{K_\lambda}, \quad \text{ for all }(z_1,z_2) \in I \times I,
$$
and
$$
\partial_{z_1} \lambda(a,a) <0 \quad \text{ and }\quad \partial_{z_1}\lambda(b,b) >0,
$$

\item[(H2)] Let $u_\ep(x,z,t) = -\ep \log n_\ep(x,z,t)$, then the initial condition satisfies $$\left\|u_\ep(x,z,0) - V_0(z) - \frac{1}{2} \ep \log \ep \right\|_{C^2(\overline{D} \times \overline{I})} = O(\ep),$$ where $V_0(z) \in C^\infty(\bar{I})$ is a non-negative function such that for some $K_0>0$, 
$$
\partial^2_z V_0(z) > {2}{K_0} \,\, \text{ for }z \in I\,\,\text{ and }\,\,{\inf_I  V_0(z) = V_0(\bar{z}_0) = 0} \quad \exists~\bar{z}_0 \in {\rm Int}~I.$$




\end{description}

\begin{remark}
{\bf(H1)} ensures that the trajectory of the dominant trait remains continuous for $t\geq 0$. 
See Section \ref{sec:dis} for an explicit example of $\alpha$ so that {\bf (H1)} can be explicitly verified\footnote{When convexity assumption fails, then in general the dominant trait has jump discontinuities and belongs to the class of BV functions
for which the uniqueness of the 
Hamilton-Jacobi equation~\eqref {eq:limithj}
is more subtle; See~\cite{Calvez2018}.}. {We leave the general case for future work.}
\end{remark}

The following {uniqueness} result, which seems to be of independent interest,  plays a  critical role in characterizing the solution trajectories of \eqref{eq:1.1}. {Its proof, which uses the convexity of $\lambda(\cdot, z_2)$, is presented  in Appendix~\ref{sec:D}.}

\begin{proposition}[Constrained Hamilton-Jacobi eq.] \label{prop:traject}
Suppose that {\rm \textbf{(H1)}}~-~{\rm \textbf{(H2)}} hold.
\begin{itemize}
\item[\rm(i)] There exists a unique viscosity solution $(V(z,t),\bar{z}(t)) \in W^{1,\infty}(D \times [0,\infty)) \times W^{1,\infty}([0,\infty))$ to the following constrained Hamilton-Jacobi equation:
\begin{equation}\label{eq:limithj}
\left\{
\begin{array}{ll}
\partial_t V + |\partial_z V|^2 - \lambda(z, \bar{z}(t))=0 &\text{ for }z \in I, \,t>0,\\
\partial_z V(z,t) = 0 &\text{ for }z \in \partial I, \, t>0,\\
V(z,0) = V_0(z) &\text{ for }z \in I,\\
\inf_{z \in I} V(z,t)  = 0 &\text{ for }t>0.
\end{array} \right.
\end{equation}

\item[\rm(ii)] $\bar{z} \in C^1([0,\infty))$ and for each $t\geq 0$, 
$V(z,t) = 0$ if and only if $z = \bar{z}(t)$. Furthermore, 
$$
V(z,t) = \frac{1}{2} \partial^2_{z}V(\bar{z}(t),t)|z - \bar{z}(t)|^2 + o(|z - \bar{z}(t)|^2)$$ and for each $T>0$ there exists $K_3>0$ such that
 ${K^{-1}_3} \leq \partial^2_{z}V(\bar{z}(t),t)  \leq  K_3$ for $ t \in [0,T]$.

\item[\rm(iii)] $\bar{z}(t)$ satisfies the ODE with coefficient $\sigma(t):= \partial^2_{z}V(\bar{z}(t),t)$
\begin{equation}\label{eq:canonical}
\frac{d}{dt} 
\bar{z}(t) 
= -\frac{\partial_{z_1} \lambda( \bar{z}(t),  \bar{z}(t))}{\sigma(t)}, 
\qquad \text{ for } t>0,
\end{equation}
with the initial data $\bar{z}(0) = \bar{z}_0$.


\end{itemize}
\end{proposition}

Next, we state our main theorem.
\begin{theorem}[Dynamics of the fittest trait]\label{thm:main}
Assume {\bf (H1)}~-~{\bf (H2)}. 
For each $T>0$, as $\epsilon \to 0$, 
\begin{equation}\label{eq:1.4.1}
n_\ep(x,z,t) \to \delta_0(z - \bar{z}(t)) \theta_{\bar{z}(t)}(x) \quad \text{ in  distribution sense in } \bar D \times \bar I \times (0,T],
\end{equation}
where $\bar{z}(t) \to \text{argmin} \,\alpha(z)$ as $t \to \infty$. 
In fact, we have
\begin{equation}\label{eq:1.4.2}
-\ep \log n_\ep(x,z,t) \to V(z,t)\quad \text{ in } C(\bar{D} \times \bar{I} \times [0,T]),
\end{equation} and
\begin{equation}\label{eq:1.4.3}
\rho_\ep(x,t) \to \theta_{\bar{z}(t)}(x) \quad \text{ in } C_{loc}(\bar D  \times (0,T]),
\end{equation}
where $V(z,t)$ and $\bar{z}(t)$ are given by Proposition {\rm\ref{prop:traject}}; and, for $z \in I$, $\theta_z(x)$ is the unique positive solution of \eqref{eq:theta}.
\end{theorem}

\begin{remark}
One can also replace the hypothesis {\bf(H1)}~-~{\bf(H2)} by  
\begin{description}

\item[(H1$'$)] Suppose $V_0(z)$ and $\alpha(z)$ are chosen such that \eqref{eq:limithj} has a unique viscosity solution $(V(z,t),\bar{z}(t)) \in W^{1,\infty}(D \times (0,\infty)) \times W^{1,\infty}(0,\infty)$ such that the conclusions (ii) and (iii) of Proposition \ref{prop:traject} hold.
\end{description}
For instance, our conclusion holds for the case $\alpha(z)>0 $ being periodic in $z$, and the Neumann condition being replaced by a periodic condition in $z$, provided that $\alpha$ has a non-degenerate minimum attained at $\hat z$, and such that $V_0(z) = K(z - \bar{z}_0)^2$, provided $|\hat z - \bar{z}_0| \ll 1$ and $K \gg 1$.
\end{remark}

\subsection{Biological interpretation}

The Darwinian evolution of a quantitative trait is the combined effects of two biological processes: (i) mutations generating variations in the trait value; and (ii) selection via relative reproductive fitness, resulting from ecological interactions between individuals and their environment. The framework of adaptive dynamics \cite{odo,Geritz2008} is based on the assumption of separation of timescales between the mutation/evolutionary and selection/ecological processes. One important advance of this theory, due to Dieckmann and Law \cite{Dieckmann1996}, is the formal derivation of the so-called canonical equation of adaptive dynamics: An ordinary differential equation that gives the rate of change over time of the expected trait value in a monomorphic population. In previous works \cite{Diekmann2005,Lorz2011}, the canonical equation was rigorously derived in case the ecological interaction can be described by ODEs, i.e., the fitness function is explicitly given in terms of the trait $z$. This is not the case, however, for evolution of dispersal, for which the incorporation of spatial structure in the model is essential and causes considerable mathematical difficulties. See \cite{Bouin2015,Jabin2016,Mirrahimi2015}. Also, using a continuous trait explains the 'accelerating waves' which have been actively studied recently, \cite{BouinH17,  BouinHR17b, BouinHR17a, HendersonPS2018}.

In this paper, we define an effective Hamiltonian $H_\epsilon(z,t)$ by improving the existing theory of the principal bundle for parabolic problems. This Hamiltonian can be viewed as a fitness function of the trait $z$ interacting with the environment at time $t$ (as described by $m(x) - \rho_\epsilon(x,t)$). One achievement of this paper is to show rigorously that, in a suitable timescale, the fitness function $H_\epsilon(z,t)$ converges to $\lambda(z, \bar{z}(t))$ in the rare mutation limit, where $\bar{z}(t)$ is the fittest trait at time $t$ (i.e. in the environment at time $t$) and, in this environment set by the fittest trait, the function $\lambda(z,\bar{z}(t))$ is the relative fitness of trait $z$ 
and is defined implicitly by an elliptic linear eigenvalue problem arising from the pairwise ecological interaction between 
two traits. The canonical equation~\eqref{eq:canonical}
can then be rigorously derived. It is interesting to note that the mean trait $\bar{z}(t)$ and the variance $\partial_{zz} V(z,t)\big|_{z = \bar{z}(t)}$ does not satisfy a closed system of ODEs. 

While our previous works \cite{Lam2017b, Lam2017a, Perthame2015} characterized the unique evolutionarily stable strategies (ESS) by solving the steady state problem, our present work on the time-dependent problem describes the approach of the dominant trait to the ESS in the evolutionary timescale.

\subsection{A heuristic presentation of the analytical approach}

The WKB-ansatz defines $u_\ep$ by
$$
n_\ep(x,z,t) = e^{ - \frac{u_\ep(x,z,t)}{\ep}}, 
$$ 
where the rate function $u_\ep(x,z,t)$ satisfies the equation
\begin{equation}\label{eq:uu}
\partial_t u_\ep - \frac{\alpha(z)}{\ep} \Delta_x u_\ep + \frac{\alpha(z)}{\ep^2}| \nabla_x u_\ep|^2 - \ep \partial^2_z u_\ep + |\partial_z u_\ep|^2 \,+ m(x) - \rho_\ep(x,t) =0.
\end{equation}
Whereas the large coefficients of the spatial derivatives suggests that $u(z,t) = \lim\limits_{\ep \to 0}u_\ep(x,z,t)$ is constant in $x$ in the limit, the equation \eqref{eq:uu} itself depends non-trivially on the spatial variables, through the terms $m(x)  - \rho_\ep(x,t)$. 

To obtain the limiting Hamilton-Jacobi equation (which is supposed to be free of the $x$-dependence), the perturbed test function method has been invented in \cite{Evans1989} by considering 
$$
n_\ep(x,z,t) = \Phi_\ep(x,z,t)e^{ - \frac{v_\ep(x,z,t)}{\ep}},
$$
i.e. $v_\ep(x,z,t) = u_\ep + \ep \log\Phi_\ep(x,z,t)$. To eliminate the $x$-dependence in \eqref{eq:uu}, we construct the corrector\footnote{The elliptic eigenfunction and eigenvalues (i.e. by omitting the term $\ep\partial_t\Phi_\ep$ in \eqref{eq:bundlee}) are not adopted here. The reason is that the lack of regularity in time for $\rho_\ep$ will then render it quite difficult to estimate
 $\ep\partial_t\Phi_\ep$, which only goes to zero in some weak, average sense.} $\Phi_\ep(x,z,t)> 0$ {as the {\it normalized principle boundle} that means to satisfy,  for each $z$,
\begin{equation}\label{eq:bundlee}
\ep\partial_t \Phi_\ep - \alpha(z)\Delta_x \Phi_\ep = \big(m(x) - \rho_\ep(x,t) + H_\ep(z,t)\big)\Phi_\ep, \qquad {\int_D \Phi_\ep(x,z,t) dx =1}.
\end{equation}
We can build (see Appendix~\ref{sec:B}) the normalizing factor $t\mapsto H_\ep(z,t)$ so that $\Phi_\ep(x,z,t)$ satisfies the mass $1$ constraint {(and is bounded in $C(\bar D \times [0,\infty))$ thanks to the Harnack inequality)}.} Denoting $\varphi_\ep:= -\log \Phi_\ep$, the equation of $v_\ep$ can be written as
\begin{equation}\label{eq:v1}
\begin{array}{ll}
\partial_t v_\ep - \frac{\alpha(z)}{\ep} \Delta_x v_\ep +  \frac{\alpha(z)}{\ep^2} |\nabla_x v_\ep|^2 +  2\frac{\alpha(z)}{\ep} \nabla_x v_\ep \nabla_x \varphi_\ep&-\ep \partial^2_z v_\ep + |\partial_z v_\ep|^2 + 2\ep \partial_z v_\ep \partial_z \varphi_\ep  \\[5pt]
\qquad \qquad\qquad  = H_\ep(z,t) - \ep^2 ( \partial^2_z \varphi_\ep + |\partial_z \varphi_\ep|^2) &\text{for }x \in D, z \in I, t>0.
\end{array}
\end{equation}
Provided that {$\varphi_\ep \in L^\infty$} and it has bounded derivatives in the trait variable\footnote{We need to show the terms $- \ep^2 ( \partial^2_z \varphi_\ep + |\partial_z \varphi_\ep|^2) \to 0$ uniformly. In general, the Cauchy problem \eqref{eq:bundlee} does not guarantee smoothness with respect to $z$, due to the dependence on initial data and the timescale $\ep t$. We will therefore make a canonical choice of an appropriate eternal solution of \eqref{eq:bundlee}. This requires {\it a priori} H\"{o}lder esimates of $\rho_\ep(x,t)$, as well as an improvement upon existing theory of principal bundle (Appendix \ref{sec:B}).}, 
we can show by comparison that $v_\ep(x,z,t) \approx \widetilde{V}_\ep(z,t)$, where the approximate solution $\widetilde{V}_\ep(z,t)$ can be obtained by solving the following Hamilton-Jacobi equation which is free of $x$-dependence:
\begin{equation}\label{eq:V1}
\partial_t \widetilde{V}_\ep  + |\partial_z \widetilde{V}_\ep|^2 = H_\ep(z,t).
\end{equation}
{Therefore, we can approximate for a given time $t$ the dominant trait, which maximizes $z\mapsto n_\ep(x,z,t)$, by the value  $\bar{z}_\ep(t)$ where $\widetilde{V}_\ep$ attains its minimum.}
In fact, we can determine this approximate trajectory $\bar{z}_\ep(t)$ from $H_\ep(z,t)$ as well. To this end, use the uniqueness of convex solutions $(V_\ep(z,t), \bar{z}_\ep(t))$ in the class $W^{1,\infty}(I \times (0,T)) \times W^{1,\infty}(0,T)$ of the following Hamilton-Jacobi equation with a constraint, due to Mirrahimi and Roquejoffre \cite{Mirrahimi2016} (for the uniqueness in $W^{1,\infty}(I \times (0,T)) \times BV(0,T)$, see also \cite{Calvez2018} and Appendix \ref{sec:D}): 
\begin{equation}\label{eq:limithjj}
\left\{
\begin{array}{ll}
\partial_t V_\ep  + |\partial_z V_\ep|^2 = H_\ep(z,t) - H_\ep(\bar{z}_\ep(t),t) &\text{ for }z \in I, \, t \in [0,T],\\
\inf_z V_\ep(\cdot,t) = 0 &\text{ for }t \in [0,T],\\
V_\ep(z,0) = V_0(z) &\text{ for }z \in I.
\end{array}
\right.
\end{equation}
Hence, we have defined the approximate trajectory $\bar{z}_\ep(t)$ and effective Hamiltonian $H_\ep(z,t)$ in terms of the quantity $\rho_\ep(x,t)$.\footnote{At this point $u_\ep(x,z,t) \approx v_\ep(x,z,t) \approx V_\ep(z,t) + \int_0^t H_\ep(\bar{z}_\ep(s),s)\,ds$. However, by the uniform positive upper and lower bounds of $\|\rho_\ep(\cdot,t)\|_{L^1(D)}$, we deduce that $\int_0^t H_\ep(\bar{z}_\ep(s),s)\,ds \to 0$ uniformly. Hence, we have $v_\ep(x,z,t) \approx V_\ep(z,t)$.}  
Now, having proved that the population has dominant trait at $\bar{z}_\ep(t)$, we can further determine that as $\ep\to 0$,
\begin{equation}\label{eq:rhoconverr}
\left| \rho_\ep(x,t) - \theta_{\bar{z}_\ep(t)}(x)\right| \to 0.
\end{equation}
Since $H_\ep(z,t)$ depends smoothly in terms of $\rho_\ep(x,t)$, \eqref{eq:rhoconverr} and the separation of timescales imply that the effective Hamiltonian $H_\ep(z,t)$ converges to the invasion exponent in \eqref{eq:invasionexp}, i.e., 
$$\left|H_\ep(z,t) -  \lambda(z,\bar{z}_\ep(t))\right| \to 0.
$$ 
Since $\lambda(z_1,z_2) = 0$ when $z_1 = z_2$, we have in fact
\begin{equation}\label{eq:Hconverr}
H_\ep(z,t) - H_\ep(\bar{z}(t),t) \approx \lambda(z,\bar{z}(t)) - \lambda(\bar{z}(t),\bar{z}(t)) = \lambda(z,\bar{z}(t)).
\end{equation}
This, in turn, allows us to pass to the limit in the constrained Hamilton-Jacobi equation \eqref{eq:limithjj}, resulting in \eqref{eq:limithj} in 
Proposition \ref{prop:traject}.

However, 
for the uniqueness argument to work on \eqref{eq:limithjj}, which is crucial in the definition of approximate trajectory $\bar{z}_\ep(t)$, we need  
$$
\partial^2_z H_\ep(\bar{z}(t),t)>0.$$
This is indeed the case in a small time interval $[0,\delta_1]$, as one can show, by the separation of time-scale between the evolutionary dynamics in $z$ and ecological dynamics of reaching the spatial equilibrium $\theta_{\bar{z}_0}$, that 
$\rho_\ep(x,t)$ does not deviate too much from $\theta_{\bar{z}_0}(x)$ in the time interval $[0,\delta_1]$. This implies the effective Hamiltonian $H_\ep(x,t)$ is in the proximity of $\lambda(z, \bar{z}_0)$, which has the relevant convexity in $z$.  This is done in Corollary \ref{cor:h2}. In this way, the argument above is valid and delivers that 
$$
H_\ep(z,t) \to \lambda(z,\bar{z}(t)) \quad \text{ and }\quad \rho_\ep(x,t) \to \theta_{\bar{z}(t)}(x) \quad \text{ for }0  \leq t  \leq  \delta_1.
$$
By carefully examining the smooth dependence of the effective Hamiltonian $H_\ep$ and corrector $\Phi_\ep$ on $z$ and $\rho_\ep$, we establish a uniform lower bound for the time-step $\delta_1$ for which the above argument can be applied. Iterating step by step in time, we can prove the convergence over the time interval $[0,T]$ for all $T>0$.

A different mutation-selection model involving a spatial variable in also studied in \cite{Jabin2016}. In the setting of that work, the rate function $u_\ep$ can be shown to be uniformly convex in $z$ {\it a priori}. This is not the case in our setting. Also, 
the quasi-steady state approximation is used in that paper, i.e., the elliptic eigenvalue problem instead of the parabolic principal bundle problem \eqref{eq:bundlee} in choosing the corrector $\Phi_\ep$. In our setting $\rho_\ep$ does not have enough {\it a priori} regularity in time, so our definition of corrector via the parabolic problem affords the needed additional time regularity.

A similar kind of result has also been obtained for the corresponding model with age structure \cite{Nordmann2020}, but with a different strategy. Therein the corrector $\Phi_\epsilon$ was defined by a couple of nonlinear mappings and the main analysis was devoted to showing the uniform boundedness of the corrector $\Phi_\epsilon$. In contrast, here we adopt a relatively direct approach by defining the corrector $\Phi_\ep$ directly to be the (bounded) 
normalized principal Floquet bundle 
for parabolic problem \eqref{eq:bundlee}. This is made possible thanks to our new {\it a priori} estimates on the quantity $\rho_\ep(x,t)$.

\subsection{Organization of the Paper}

In Section \ref{sec:2}, we establish some {\it a priori} estimates of the solution $n_\ep(x,z,t)$ and its integrated version $\rho_\ep(x,t)$. We first state the global positive upper and lower bounds, for which the proofs are postponed to Appendix \ref{sec:A}, and derive the H\"{o}lder regularity of $\rho_\ep$ and $\partial^2_{x} n_\ep$ in some appropriately rescaled variables.

In Section \ref{sec:3}, we apply the theory of normalized principal Floquet bundle for parabolic problems with Neumann boundary conditions, as developed in \cite{Cantrell2021evolution} and summarized in Appendix \ref{sec:B}, to define the effective Hamiltonian $H_\ep(z,t)$ and the corrector $\Phi_\ep(x,z,t)$, as functions of $\rho_\ep(x,t)$. We also state the uniqueness theorem, which allows us to define the approximate trajectory $\bar{z}_\ep(t)$ and pheonotypic distribution $V_\ep(z,t)$ in terms of the effective Hamiltonian $H_\ep(z,t)$ and initial distribution $V_0(z)$.

In Section \ref{sec:4}, we introduce the rate function $u_\ep(x,z,t)  = -\ep \log n_\ep(x,z,t)$ and prove several technical estimates that enable us to obtain a lower bound $\delta_1>0$ of the step size in time with which we can continue the approximate trajectory.

In Section \ref{sec:main}, we prove the main result
Theorem \ref{thm:main}.

The three appendices are devoted to the three main ingredients/tools that we develop for this singularly perturbed problem. In Appendix \ref{sec:A}, we prove the {\it a priori} $L^\infty$ estimate of $\rho_\ep(x,t)$. In Appendix \ref{sec:B}, we state the existence and differentiability of the normalized principal Floquet bundle, which is used in the construction of effective Hamiltonian and corrector $(H_\ep(z,t), \Phi_\ep(x,z,t))$.
In Appendix \ref{sec:D}, we prove the existence and uniqueness of solutions to the constrained Hamilton-Jacobi equations that we need.

\section{A priori Estimates}\label{sec:2}

We begin with the following result; see 
Proposition \ref{prop:apriori}
for the proof: 
\begin{proposition}\label{prop:apriorisum}
Let $\rho_\ep(x,t) = \int_I n_\ep(x,z,t)\,dz$, where $n_\ep$ is the solution to \eqref{eq:1.1}. Then
there exists $\widehat{C}_1$ independent of $\epsilon>0$ such that
$$\frac{1}{\widehat{C}_1} \leq \rho_\ep(x,t) \leq \widehat{C}_1\quad \text{ for all }\,\,(x,t) \in D \times [0,\infty).
$$
\end{proposition}

In the following, we extend $n_\ep(x,z,t)$ evenly and then periodically in
$z$. Due to the Neumann boundary condition, the extended $n_\ep$ satisfies the same PDE with coefficients similarly extended to $D \times  \mathbb{R} \times [0,\infty)$. Next, 
we rescale $n_\ep$. Define, for $z_1 \in I$ and $t_1 \geq \ep$, 
$$
N_\ep(x,y,\tau)=N_\ep(x,y,\tau; z_1, t_1) := n_\ep(  x, z_1 + \ep y,t_1 + \ep \tau),
$$
and note that $N_\ep(x,y,\tau)$ satisfies a linear parabolic equation 
\begin{equation}\label{eq:Nep}
\partial_\tau N_\ep - \alpha(z_1 + \ep y) \Delta_x N_\ep - \partial^2_y N_\ep = N_\ep \big(m(x) - \rho_\ep(x, t_1 +\ep \tau)\big).
\end{equation}
By Proposition \ref{prop:apriorisum}, the above equation has $L^\infty$ bounded coefficients, so we may apply local parabolic $L^p$ estimates to obtain H\"{o}lder regularity of $N_\ep$, which then allows us to use parabolic Schauder estimates to estimate $\partial^2_x N$. Here $\partial^2_x N$ denotes all second order partial derivatives in space $\partial^2_{x_i x_j} N$. 

\begin{lemma}\label{lem:localholder}
For $\beta \in (0,1)$,  there exists $C>0$ independent of $\ep$ such that
$$
\|\partial^2_x N_\ep(x,y,\tau)\|_{C^{\beta, \beta, \beta/2 }(  \bar D  \times[-\frac{1}{2},\frac{1}{2}] \times [-\frac{1}{2},0])} \leq C \|N_\ep\|_{L^1( D \times (-1,1) \times (-1,0) )}.
$$
In particular, there exists a constant $C$ independent of $t \geq \ep$ and $z\in I$ such that
\begin{equation}\label{eq:localholder2}
\sup_{x \in D} |\partial^2_x n_\ep(x,z,t)| \leq C   \fint_{t-\ep}^{t}\fint_{-\ep}^{\ep} \int_D n_\ep(x',z+z',t')\,dx'dz'dt'\,.
\end{equation}
Here we use the notation $\fint_{s_1}^{s_2}  = \frac{1}{s_2-s_1} \int_{s_1}^{s_2}$, for any $s_1 < s_2$.
\end{lemma}
\begin{proof}
By Proposition \ref{prop:apriorisum}, $\displaystyle  \sup_{t \geq 0} \|\rho_\ep(\cdot,t)\|_{C(\bar D)} \leq C$. Hence, the equation \eqref{eq:Nep} has $L^\infty$ bounded coefficients. So we may apply local parabolic $L^p$ estimates to obtain H\"{o}lder regularity of $N_\ep$. To this end, 
define $\Omega_{3/5} \subset \Omega_{4/5} \subset \Omega_1$ by 
$$
\Omega_R :=   D \times \left(-R,R\right)\times \left(-R,0\right) , \quad \text{ for }R = \frac{3}{5}, \frac{4}{5}\,\text{ and }\, 1.
$$ Then, recalling \eqref{eq:Nep}, we have 
\begin{equation}\label{eq:Nstep2}
\|{{N_\ep}}\|_{C^{\beta, \beta,\beta/2}(\overline{ \Omega_{3/5}})} \leq C \|N_\ep||_{W^{2,2,1}_{p}(\Omega_{3/5})} \leq C \|N_\ep\|_{L^p(\Omega_{4/5})} \leq C \|N_\ep\|_{L^1(\Omega_{1})},
\end{equation}
where the first, second and third inequality follows, respectively, from the Sobolev embedding, parabolic $L^p$ estimate, and \eqref{eq:localmaxprin}.

Integrating \eqref{eq:Nstep2} over $z_1 \in I$, we have, for each $t_1 \geq \ep$,
\begin{align}
\|\rho_\ep(\cdot~,t_1 + \ep~\cdot~)\|_{W^{2,1}_p( {D}\times \left[ -\frac{3}{5},0\right])} &\leq \int_I \|N_\ep(\cdot,0,\cdot \,; t_1,z_1)\|_{W^{2,1}_p( {D} \times [ -\frac{3}{5},0] )}  \,dz_1  \notag\\ 
&\leq \int_I \|N_\ep(\cdot,0,\cdot\,;t_1,z_1)\|_{L^1(D \times [-1,0])}\,dz_1 \notag\\
&\leq C\| \rho_\ep(\cdot, t_1 + \ep \cdot)\|_{L^1(D \times [-1,0])}\leq C \label{eq:w21p}
\end{align}
where we used the fact that $\rho_\ep(x,t_1+\ep t) = \int_I N_\ep(x,0,t;t_1,z_1)\,dz_1$ for the first and third inequalities, \eqref{eq:Nstep2} in the second inequality, and Proposition \ref{prop:apriori} for the last inequality. 
Similarly, for each $t_1 \geq \ep$ we have
$$
\|\rho_\ep(\cdot~,t_1 + \ep~\cdot~)\|_{C^{\beta, \beta/2}( \overline{D}\times \left[ -\frac{3}{5},0\right])} \leq \int_I \|N_\ep(\cdot,0,\cdot ; t_1,z_1)\|_{C^{\beta, \beta/2}( \overline{D} \times [ -\frac{3}{5},0] )}  \,dz_1 \leq C.
$$

Now that $\rho_\ep(\cdot, t_1 + \ep \cdot)$ is H\"{o}lder continuous, we may apply parabolic Schauder estimates to \eqref{eq:Nep} to get
\begin{equation}\label{eq:Nstep4}
\|\partial^2_x N_\ep(x,y,\tau)\|_{C^{\beta, \beta,\beta/2}\left(\overline{D}\times \left[ -\frac{1}{2},-\frac{1}{2}\right]\times \left[ -\frac{1}{2},0\right]\right)} \leq C \|N_\ep\|_{C^{\beta, \beta,\beta/2}(\overline{\Omega_{3/5}})}.
\end{equation}
The lemma follows from combining \eqref{eq:Nstep2} and \eqref{eq:Nstep4}.
\end{proof}

The following result follows from the proof of Lemma \ref{lem:localholder}: 
\begin{corollary}\label{cor:apriorisum}
\begin{itemize}
\item[{\rm(a)}] For each $p>1$, there exists some  $\widehat{C}_p>0$ independent of $\epsilon$ such that 
$$
\sup_{\tau_0 \geq 1} \|\widetilde\rho_\ep(\cdot,\cdot)\|_{W^{2,1}_p(D \times [\tau_0,\tau_0+1])} \leq \widehat{C}_p, \quad \text{ where }\quad \widetilde\rho_\ep(x,\tau) := \rho_\ep(x,\epsilon \tau). 
$$

\item[{\rm(b)}] For each $\beta'\in (0,1)$, there exists some  $C_{\beta'}>0$
independent of $\epsilon$ such that
$$
\|\widetilde\rho_\ep\|_{C^{\beta',\beta'/2}(\bar D \times [1,\infty))} \leq C_{\beta'}.
$$ 
In particular, we have $\sup\limits_{t \geq \ep} \|\rho_\ep(\cdot,t)\|_{C^{\beta'}(\bar D)} \leq C_{\beta'}.$
\end{itemize}
\end{corollary}

This result requires an initial delay of order $\ep$ so as to take into account the possible initial layer on $\rho_\ep$. This is responsible for the technical issues on the initial data that we encounter in the next section.

\section{Approximate Trajectory via the Normalized Principal Floquet Bundle}\label{sec:3}

Our next and fundamental task is to define the corrector $\varphi_\ep$, closely related to $\log \Phi_\ep$ and effective Hamiltonian $H_\ep$ in terms of the principal bundle of certain parabolic problem with potential $m - \rho_\ep$. This, in turn, enables us to define an approximate trajectory $\bar{z}_\ep(t)$ of the dominant trait.


\begin{proposition}\label{prop:A3}
For each fixed $\ep>0$ and $z \in I$, there exists a unique classical solution $(\varphi_\ep(x,z,t),H_\ep(z,t))$ to the following linear parabolic problem in $D\times (-\infty,\infty)$: 
\begin{equation}\label{eq:hamiltonian}
\left\{
\begin{array}{ll}
\ep \partial_t \varphi_\ep - \alpha(z)
\Delta_x
\varphi_\ep + \alpha(z)|
\nabla_x\varphi_\ep|^2 + m(x) - \rho_\ep(x,\max\{t,\ep\})&\\
  \qquad  + H_\ep(t;z)=0 &  \hspace{-1cm}\quad x \in D, t\in \mathbb{R},\\
\partial_\nu \varphi_\ep = 0 &\hspace{-1cm}\quad x \in \partial D, t\in \mathbb{R},\\
\int_D e^{- \varphi_\ep(x,t;z)}\,dx = 1 &\hspace{-1cm}\quad  t\in  \mathbb{R}.
\end{array}
\right.
\end{equation}
Moreover, the quantities $(\varphi_\ep(x,t;z), H_\ep(z,t))$ depend smoothly on $z \in I$, i.e., for some constant $\widetilde{C}_0$ independent of $\ep$, 
\begin{equation}\label{eq:uniformbounds}
\max\limits_{i=0,1,2,3} \|\partial_z^iH_\ep(t;z)\|_\infty + \|
\nabla_x\varphi_\ep(x,t;z)\|_\infty + \max_{i=0,1,2}\|\partial^i_z \varphi_\ep(x,t;z)\|_\infty  \leq \widetilde{C}_0,
\end{equation}
where $\| \cdot \|_\infty$ denotes the $L^\infty$ norm over $(x,z,t) \in D \times I \times \mathbb{R}$. 
\end{proposition}
\begin{proof}
For $x \in D$ and $\tau \in \mathbb{R}$, define  $\displaystyle
c_\ep(x,\tau) := m(x) - \widetilde{\rho}_\ep(x, \max\{\tau,1\})$, 
where  $\widetilde\rho_\ep(x,\tau) = \rho_\ep(x,\epsilon \tau)$ as in Corollary~\ref{cor:apriorisum},
then $\|c_\ep\|_{C^{\beta,\beta/2}(\bar D \times \mathbb{R})}$ is uniformly bounded in $\ep$ thanks to  Corollary~\ref{cor:apriorisum}(b). By Theorem~\ref{thm:A1}, we can define the corresponding normalized principal Floquet bundle $$(\Phi_1(x,\tau;c_\ep,z), H_1(\tau;{c_\ep},z))  \in C^{2+\beta,1+\beta/2}(\bar D\times \mathbb{R})\times C^{\beta/2}(\mathbb{R}),$$
which satisfies
$$
\begin{cases}
\partial_\tau \Phi_1 - \alpha(z) \Delta_x \Phi_1 = (m(x) - \rho_\ep(x,\epsilon\max\{\tau,1\})\Phi_1  & \text{ for }x \in D,~t\in\mathbb{R},\\
\partial_\nu \Phi_1 = 0 &\text{ for }x \in \partial D,~ t \in \mathbb{R},\\
\Phi_1 >0  \quad  \text{ for }x \in D,~t\in\mathbb{R},\quad \text{ and }\quad 
\int_D \Phi_1\,dx = 1 &\text{ for }t\in\mathbb{R}.
\end{cases}
$$
Setting
$$
\varphi_\ep(x,z,t) := -\log \Phi_1(x,t/\ep;c_\ep,z) \quad \text{ and }\quad H_\ep(z,t) := H_1(t/\ep;c_\ep,z),
$$
we obtain $(\varphi_\ep,H_\ep)$ satisfying \eqref{eq:hamiltonian}.
The smoothness follows from Proposition~\ref{prop:A2}. 
\end{proof}

We define the approximate trajectory by solving a constrained Hamilton-Jacobi equation by making use of uniqueness results under convexity assumption~\cite{Mirrahimi2016}. The proof is contained in Appendix \ref{sec:D}.

\begin{proposition}\label{prop:traject2}
Suppose, for some $T>0$, it holds that
\begin{equation}\label{eq:3.6.2}
\liminf_{\epsilon \to 0} \left[\inf_{I \times [2\sqrt\ep, T]} 
\partial^2_{z}H_\ep(z,t) \right]>0.
\end{equation} 
Then for all $\epsilon>0$ small, there exists a unique viscosity solution $(V_\ep(z,t),\bar{z}_\ep(t))$ to the following Hamilton-Jacobi equation with a constraint:
\begin{equation}\label{eq:approxhj}
\hspace{-.5cm}\left\{
\begin{array}{ll}
\partial_t V_\ep + |\partial_z V_\ep|^2 - H_\ep(z,t) + H_\ep(\bar{z}_\ep(t),t) =0&\text{ for }t \in [2\sqrt\ep,T],\,z \in I, \\
\partial_z V_\ep =0 &\text{ for }t \in [2\sqrt\ep,T],\,z \in \partial I,\\
{ \bar{z}_\ep({2\sqrt{\ep})}} = \bar{z}_0,\quad \text{ and }\quad V_\ep(z,2\sqrt\ep) = V_0(z) &\text{ for }z \in I,\\
\inf_{z \in I} V_\ep(z,t) = 0 &\text{ for }t\in [{2\sqrt\ep},T].
\end{array} \right.
\end{equation}
Moreover, 
$$
V_\ep(\bar{z}_\ep(t),t)=0, \quad\text{ and }\quad V_\ep(z,t) >0 \quad \text{ for }z \neq \bar{z}_\ep(t),
$$
and 
there exists a constant $C>0$ independent of $\epsilon$ such that 
\begin{equation}\label{eq:liptraj}
\|\bar{z}_\ep(\cdot)\|_{C^{0,1}([2\sqrt\ep,T])} \leq C.
\end{equation}
\end{proposition}
\begin{remark}
Thanks to \eqref{eq:3.6.2} and the convexity of the initial data, $V_\ep$ and $\bar{z}_\ep(t)$ satisfy \eqref{eq:approxhj}, except for the Neumann boundary condition, in the classical sense. The Neumann boundary condition has to be understood in the classical sense \cite{Barles2013introduction}.
\end{remark}

\section{WKB Ansatz and Some Technical Lemmas}\label{sec:4}

The rate function $u_\ep(x,z,t) = - \ep \log n_\ep(x,z,t)$ satisfies the equation
\begin{equation}\label{eq:u}
\left\{
\begin{array}{ll}
\partial_t u_\ep - \frac{\alpha(z)}{\ep}
\Delta_x u_\ep + \frac{\alpha(z)}{\ep^2}| 
\nabla_xu_\ep|^2 - \ep \partial^2_z u_\ep + |\partial_z u_\ep|^2 \,+& m(x) - \rho_\ep(x,t) =0\\
& x \in D, z \in I, t>0,\\
\partial_\nu u_\ep = 0 & x \in \partial D, z \in I, t>0,\\ 
\partial_z u_\ep = 0 & x \in D, z \in \partial I, t>0,\\
u_\ep(x,z,0) = u_{\ep,0}(x,z) := -\ep \log n_{\ep,0}(x,z)
& x \in D, z \in I.
\end{array}
\right.
\end{equation}

Let $ \varphi_\ep(x,z,t)$ be given in Proposition~\ref{prop:A3}.
Using the perturbed test function method, we define the function $v_\ep(x,z,t):= u_\ep(x,z,t) - \ep \varphi_\ep(x,z,t)$,
which satisfies
\begin{equation}\label{eq:v}
\left\{
\begin{array}{ll}
\partial_t v_\ep - \frac{\alpha(z)}{\ep} \Delta_x v_\ep +  \frac{\alpha(z)}{\ep^2} |\nabla_x v_\ep|^2 +  2\frac{\alpha(z)}{\ep} \nabla_x v_\ep \cdot \nabla_x \varphi_\ep - \ep \partial^2_z v_\ep &+|\partial_z v_\ep|^2+ 2\ep \partial_z v_\ep \partial_z \varphi_\ep  \\
= H_\ep(z,t) - \ep^2 ( \partial^2_z \varphi_\ep + |\partial_z \varphi_\ep|^2) {+ \rho_\ep(x,t) - \rho_\ep(x,\max\{t,\ep\})} &\text{ for }x \in D, z \in I, t>0,\\
\partial_\nu v_\ep = 0 &\text{ for }x \in \partial D, z \in I, t>0,\\
\partial_z v_\ep = - \ep \partial_z \varphi_\ep = O(\ep) &\text{ for }x \in D, z \in \partial I, t>0,\\
v_\ep(x,z,0) = u_{\ep,0}(x,z) &\text{ for }x \in D, z \in I.
\end{array}\right.
\end{equation}
\begin{remark}
The right hand side of \eqref{eq:v} is essentially the Hamiltonian $H_\ep(z,t) + O(\ep)$, since 
the term $\rho_\ep(x,t) - \rho_\ep(x, \max\{t,\ep\})$ is identically zero except for $t \in [0,\ep]$.
\end{remark}


By \eqref{eq:uniformbounds} we have
$$
\left\|u_\ep - v_\ep\right\|_{L^\infty( D \times I \times (0,\infty))} = \ep\| \varphi_\ep \|_{L^\infty( D \times I \times (0,\infty))} = O(\ep),
$$ 
thus upper and lower estimates of $u_\ep$ and $v_\ep$ are the same, up to an error of $O(\ep)$. In the following, we shall construct super- and sub-solutions of $v_\ep$.

\begin{proposition}\label{prop:perturbation}
Suppose for some $\hat{t}= \hat{t}_\ep\geq \ep$, 
$\hat{z}\in I$, $V_1 \in C^2(\bar I)$ and $\eta_1>0$, such that
\begin{equation}\label{eq:convexity}
V_1(z)  \geq 0, \quad V_1(z) = 0 \text{ iff }z=\hat{z},\quad \partial^2_z V_1(z) > \eta_1 >0 \text{ for }\hat{z} - \eta_1\leq  z \leq \hat{z} + \eta_1
\end{equation}
and it holds that
$$
\sup_{\hat{t}-\ep \leq t \leq \hat{t}}\|u_\ep(x,z,t) - V_1(z)\|_{C(\bar D \times \bar I)} \to 0 \quad \text{ as }\ep \to 0.
$$
Then there exists $C>0$ independent of  $\hat{t}$ and $\hat{z} \in \bar I$ such that for each small $\delta>0$,
\begin{equation}\label{eq:prop:perturbation.1}
\limsup_{\ep \to 0}\sup_{\hat{t} \leq t \leq \hat{t} + 2\delta}\left\| \int_I (z - \hat{z}) \partial^2_x n_\ep(\cdot,z,t)~dz\right\|_{C(\bar D)} \leq C\delta^{1/3}.
\end{equation}
\end{proposition}
\begin{proof}
Note that, similar to the argument in Step 3 of the proof of Proposition
\ref{prop:apriori},
{a lower solution for \eqref{eq:u} can be constructed as:
\begin{equation}\label{eq:p10a.2}
\underline{U} (z,t)=  {V}_1(z) - C_1( \mu_\ep + t - (\hat t - \ep)) - \frac{C}{\ep}\left\{[ z - (b - \sqrt\ep)]^3_+ + [ a + \sqrt\ep -z]^3_+\right\} ,
\end{equation}
where $a=\inf I$, $b=\sup I$, and }
$$
\mu_\ep = \sup\limits_{x\in D, z \in I} |u_\ep(x,z,\hat{t}-\ep) - V_1(z)|
$$ 
is a constant tending to zero as $\ep \to 0$, and that   
$C_1 = C_1(\|m(x) -\rho_\ep\|_{C(\bar D \times [0,\infty))})$. 
By comparison in $D \times I \times [\hat t - \ep, +\infty)$, we have 
$$
u_\ep(x,z,t) \geq \sup_{|z - \hat z|\geq \delta} V_1 - C_1(\mu_\ep + |t-\hat{t} + \ep|) + O(\sqrt\ep) \quad \text{ for } \, |z - \hat{z}|  \geq \delta~\text{ and }~t>\hat{t}-\ep,
$$
where 
$$
\sup\limits_{|z - \hat z| \geq \delta} V_1 \geq \frac{\eta_1 \delta^2}{2} \quad \text{ for } \quad 0<\delta \ll 1.
$$ 
Therefore, for each fixed $\delta$ small, by choosing $\epsilon$ sufficiently small, 
\begin{equation*}
u_\ep(x,z,t) \geq \frac{\eta_1\delta^2}{3} \quad \text{ for }|z - \hat{z}| > \delta ~\text{ and }~ \hat{t}-\ep\leq t \leq \hat{t} + \delta^3,
\end{equation*}
which means
\begin{equation}\label{eq:prop:perturb.1}
n_\ep(x,z,t) \leq \exp\left( -\frac{\eta_1\delta^2}{3\epsilon}\right) \quad \text{ for }|z - \hat{z}| > \delta ~\text{ and }~ \hat{t}-\ep\leq t \leq \hat{t} + \delta^3.
\end{equation}

Now, for $t_1 \in [\hat{t}, \hat{t}+\delta^3]$, 
\begin{align*}
\Big\| \int_I& (z_1 - \hat{z}) \partial^2_x n_\ep(\cdot,z_1,t_1)~ dz_1\Big\|_{C(\bar D)} \leq \int_I |z_1 - \hat{z}|\left\|  \partial^2_x n_\ep(\cdot,z_1,t_1)\right\|_{C(\bar D)} \,dz_1 \\
& \leq \int_I |z_1 - \hat{z}|  \left[ \fint_{t_1-\ep}^{t_1}\fint_{-\ep}^{\ep} \int_D n_\ep(x,z+z_1,t)\,dxdzdt  \right] \,dz_1   \\
&\leq C\delta \int_{\hat{z}- 2\delta}^{\hat{z}+ 2\delta}  \left[ \fint_{t_1-\ep}^{t_1}\fint_{-\ep}^{\ep} \int_D n_\ep(x,z+z_1,t)\,dxdzdt  \right] \,dz_1  \\
&\qquad + C \int_{z_1: |z_1 - \hat{z}| > 2\delta} \left[ \fint_{t_1-\ep}^{t_1}\fint_{-\ep}^{\ep} \int_D n_\ep(x,z + z_1,t)\,dxdzdt  \right] \,dz_1 \\
&\leq C\delta  \fint_{-\ep}^{\ep} \left[ \fint_{t_1-\ep}^{t_1}\int_{\hat{z}- 2\delta}^{\hat{z}+ 2\delta} \int_D n_\ep(x,z+z_1,t)\,dxdz_1dt  \right] \,dz+ \exp\left(-\frac{\eta_1\delta^2}{3\ep} \right)\\
&\leq 
C'  \delta  \fint_{-\ep}^{\ep} \left[ \sup\limits_{\hat t-\ep < t < \hat t} \|\rho_\ep(t,\cdot)\|_{L^1(D)} \right]\,dz + \exp\left(-\frac{\eta_1\delta^2}{3\ep} \right),
\end{align*}
{where we used \eqref{eq:localholder2} in Lemma \ref{lem:localholder} in the second inequality, and we switched the order of integration and used \eqref{eq:prop:perturbation.1} to obtain the fourth inequality.}
By fixing $\delta$ to be small enough, we see that \eqref{eq:prop:perturbation.1} holds for all $\epsilon$ sufficiently small. This completes the proof.
\end{proof}

\begin{lemma}\label{lem:Lp}
Consider the equation
$$
\left\{\begin{array}{ll}
\ep \partial_t \rho - \alpha(z_1(t)) 
\Delta_x\rho = \rho(m(x) - \rho) + F(x,t) &\text{ for }~x \in D,~t_1 \leq t \leq t_2,\\
\partial_\nu \rho = 0 &\text{ for }~x \in \partial D,~t_1 \leq t \leq t_2.
\end{array}\right.
$$
Suppose, for some fixed constant $M>1$, 
$$
z_1(t) \in I ,\quad |z'_1(t) | \leq M \quad \text{ for } ~t_1 \leq t \leq t_2.
$$
Then for each $p>1$, there exists $C=C(p,M)>1$ such that for $t \in [t_1,t_2]$, 
\begin{align*}
&\|\rho(\cdot,t) - \theta_{z_1(t)}(\cdot)\|_{L^p(D)}^p \\
& \quad \leq C \left[\|\rho(\cdot,t_1) - \theta_{z_1(t_1)}(\cdot)\|^p_{L^p(D)} e^{-(t-t_1)/(\ep C_p)} +  \sup_{t_1 \leq t' \leq t} \|F(\cdot,t')\|_{L^p(D)}^p + \ep\right].
\end{align*}
\end{lemma}
\begin{proof}
For each $t$, set  $\alpha_1(t):=\alpha(z_1(t))$ and let $\theta_t(x)=\theta_{z_1(t)}(x)$ denote the unique positive solution of
$$
\alpha_1(t) \Delta_x\theta + \theta (m - \theta) = 0 \quad\text{ in }D,~\quad ~\partial_\nu \theta =0 \quad \text{ on }\partial D.
$$
Set $W(x,t):=\rho(x,t)/ \theta_{t}(x) -1$, it satisfies the Neumann boundary condition in $x$, i.e.,
$$
\partial_\nu W(x,t) = 0 \quad \text{ for }t \geq 0, \, x \in \partial D,
$$
and, for $p>1$,  the equation
\begin{align*}
&\frac{\ep}{p} \partial_t (\theta_t^2|{W}|^p) - \alpha_1(t) 
\nabla_x\cdot \left(\theta_t^2 |{W}|^{p-1} \nabla_x
{W}\right) +  \alpha_1(t)\frac{4(p-1)}{p^2}\theta^2_t |\nabla_x
W^{p/2}|^2   + \rho \theta^2_t |W|^p \\
&\qquad = F{\theta_t}W |W|^{p-2} - \ep \theta_t \partial_t \theta_t W |W|^{p-2}  - \ep \left( 1 - \frac{2}{p}\right)  \theta_t \partial_t \theta_t |W|^{p} ,
\end{align*}
from which we derive the differential inequality
\begin{align*}
&\frac{\ep}{p} \partial_t (\theta_t^2|{W}|^p) - \alpha_1(t) 
\nabla_x\cdot\left(\theta_t^2 |{W}|^{p-1} \nabla_x
{W}\right)  + \rho \theta^2_t |W|^p \\
&\qquad \leq    \theta_t |{W}|^{p-1} |F| + \hat{C}_0 \ep \theta_t|\partial_t \theta_t| ( 1 + |W|^p)\\
&\qquad \leq \frac{1}{3\hat{C}_1} \theta_t^2 |W|^p + (3\hat{C}_1)^{p-1}  \theta_t^{2-p} |F|^p + \frac{1}{3\hat{C}_1} \theta_t^2 |W|^p + C \ep
\end{align*}
for $0<\epsilon \ll 1$, where $\hat{C}_1$ is given in Proposition \ref{prop:apriorisum}, and (after enlarging $C$ if necessary) we may assume that the constant $C$ is independent of $\ep$ and that
$$
 \frac{1}{C}\leq  \theta_t(x) \leq C \,\,\text{ for }x \in D,\quad \text{ and }\quad  \sup_{z \in \bar I}\|\theta_t\|_{C^2(\bar D)}\leq C.
$$
Integrating in $x \in D$, we obtain
$$
\frac{\epsilon}{p} \frac{d}{dt}  \int_D \theta^2_t |{W}|^p~dx  + \frac{1}{3\widetilde{C}_1} \int_D {\theta_t^2} |{W}|^p~dx \leq C\left(\int_D |F(x,t)|^p~dx + \ep\right).
$$
Let $M(t):= \int_D \theta^2_t |{W}|^p~dx$, we integrate the above 
in $t \in [t_1,t_2]$ to get
%
$$
M(t_2)  \leq M(t_1) e^{-C_p(t_2 - t_1)/\ep} + C_p \left[ \sup_{t_1 \leq t \leq t_2} \|F(\cdot,t)\|_{L^p(D)}^p + \ep \right].
$$
The lemma follows because of the uniform boundedness of $\theta_t$ from the above and below.
\end{proof}

\begin{proposition}\label{prop:h2}
Given $T>0$ and let $\|\bar{z}_\ep(t)\|_{C^{0,1}([\sqrt\ep,T])}$ be uniformly bounded in $0<\ep\ll 1$. For each $\eta>0$, there exists $\nu = \nu(\eta)>0$ such that if 
\begin{equation}\label{eq:proph2a}
\limsup_{\ep \to 0}\sup_{[\sqrt\ep,T]} \|\rho_\ep(\cdot,t) - \theta_{\bar{z}_\ep(t)}(\cdot)\|_{L^2(D)}  < \nu,
\end{equation}
then 
\begin{equation}\label{eq:proph2c}
\limsup_{\ep \to 0} \sup_{[2\sqrt\ep,T - \sqrt\ep]} \|H_\ep(\cdot,t) - \lambda(\cdot,\bar{z}_\ep(t))\|_{C^2(\bar I)} < \eta.
\end{equation}
%

\end{proposition}

\begin{proof}
By Proposition \ref{prop:A2} in Appendix \ref{sec:B}, the principal bundle depends smoothly on the weight function and $z \in I$, i.e. the mapping
$$
\begin{cases}
H_1: C^{\beta/2,\beta}( \mathbb{R}\times \bar{D}) \times \bar I \to C^{\beta/2}(\mathbb{R})\\
(c(x,t),z) \mapsto H_1(\tau)
\end{cases}
$$
is smooth. Recall also the fact, when $c(x,t)= m(x)-\theta_{\hat z}(x)$ 
$$
H_1(\tau)= H_1(\tau; m - \theta_{\hat z}, z) \equiv  \lambda(z,\hat z) \quad \text{ and is constant in }\tau \in \mathbb{R}.
$$ 
Therefore, for each $\eta>0$ there exists $\nu'>0$ such that if 
\begin{equation}\label{eq:proph2proof1}
\|c(x,\tau) - (m(x)-\theta_{\hat z}(x))\|_{C^{\beta,\beta/2}(\bar D \times \mathbb{R})} < \nu' \quad \text{ for some }{\hat{z}} \in \bar I,
\end{equation} then the corresponding principal bundle $H_1(\tau; c,z)$ satisfies
\begin{equation}\label{eq:proph2b}
 \|H_1(0; c,\cdot ) - \lambda(\cdot,\hat z)\|_{C^2(\bar I)}  \leq  \sup_{\tau \in \mathbb{R}} \|H_1(\tau; c,\cdot ) - \lambda(\cdot,\hat z)\|_{C^2(\bar I)} < \eta.
\end{equation}
Now, fix $\beta' \in (\beta,1)$, and choose by interpolation the constant $\nu>0$ such that if 
\begin{equation}\label{eq:proph2d}
\left\{
\begin{array}{l}
\|c(x,\tau) - [m(x)-\theta_{\hat z}(x)]\|_{C^{\beta',\beta'/2}(\bar{D}\times \mathbb{R})}\leq \hat{C}_{\beta'} + \|\theta_{\hat z}(x) - m(x)\|_{C^{\beta'}(\bar D)}\\
\sup_{\tau \in \mathbb{R}}  \|c(x,\tau) - [m(x)-\theta_{\hat z}(x)]\|_{L^2(D)} < 2\nu,
\end{array}\right.
\end{equation}
then \eqref{eq:proph2proof1} and thus \eqref{eq:proph2b} holds.
(Note that $\nu>0$ depends on the uniform bound $\hat{C}_{\beta'}$ in Corollary \ref{cor:apriorisum} but is independent of $\ep, \hat z$.)

We claim that if \eqref{eq:proph2a} holds for the constant $\nu$ we just specified, then \eqref{eq:proph2c} holds. Suppose not, then 
there exist sequences $\ep =\ep_j \to 0$ and $t_j \in [2\sqrt{\ep_j}, T -  \sqrt{\ep_j}]$ such that $\bar{z}_{\ep_j}(t_j)\to \hat z$ and 
\begin{equation}\label{eq:proph2e}
\left\{\begin{array}{l}
\lim\limits_{\ep_j \to 0} \sup\limits_{[\sqrt{\ep_{j}},T]} \|\rho_ {\ep_j} (\cdot,t) - \theta_{\bar{z}_{\ep_j}(t)}(\cdot)\|_{L^2(D)} \leq  \nu, \\
\lim\limits_{\ep_j \to 0}\|H_{\ep_j}(\cdot, t_j)  -  \lambda(\cdot,\bar{z}_{\ep_j}(t_j))\|_{C^2(\bar I)} =\lim\limits_{\ep_j \to 0}\|H_{\ep_j}(\cdot, t_j)  -  \lambda(\cdot, \hat z)\|_{C^2(\bar I)} \geq \eta.
\end{array}\right.
\end{equation}
Let $(\varphi_\ep(x,z,t), H_\ep(z,t))$ be the function given by Proposition \ref{prop:A3}, and define
$$
\Phi_{j}(x,z,\tau):= e^{- \varphi_{\ep_j}(x,z, t_j + \ep_j \tau)},\quad \text{ and }\quad H_{j}(z,\tau) := H_{\ep_j}(z, t_j + \ep_j\tau).
$$
Passing to a subsequence if necessary, we deduce that 
$$
\Phi_{j}(x,z,\tau) \to \hat\Phi(x,z,\tau) \,\, \text{ in }C_{loc}(\bar{D}\times \bar I \times\mathbb{R}),\, \,\text{ and }\,\, H_{j}(z,\tau) \to \hat H(z,\tau) \,\,\text{ in }C_{loc}(\bar I \times\mathbb{R}),
$$
where the limit $(\hat\Phi( x,z,\tau),\hat H(\tau))$ is a pair of the normalized Floquet principal bundle, i.e.
$$
(\hat\Phi( x,z,\tau),\hat H(z,\tau)) = (\Phi_1(x,\tau;\hat c, z), H_1(x,\tau; \hat c, z)),
$$
for some
$$
\hat c(x,\tau) = \lim\limits_{\ep_j \to 0} (\rho_{\ep_j}(x,t_j + \ep_j \tau) - m(x)). 
$$
By the first statement of \eqref{eq:proph2e}, we deduce that the weight $\hat c$ satisfies \eqref{eq:proph2d}, and hence the limit $\hat H$ of $H_{\ep_j}$ (in the $C_{loc}(\mathbb{R})$ topology) satisfies \eqref{eq:proph2b}. 
Moreover, recalling \eqref{eq:uniformbounds}, we have 
$$
\sup_j  \|H_j(\cdot,0)\|_{C^3(\bar I)} 
{ =} 
\sup_j \|H_{\ep_j}(\cdot, t_j)\|_{C^3(\bar I)} \leq C, 
$$
we deduce that $H_{j}(\cdot,0) \to \hat{H}(\cdot,0)$ in $C^2(\bar I)$. Hence
$$
\lim_{\ep_j \to 0} \| H_{\ep_j}(\cdot,t_j) - \lambda(\cdot,\hat{z})\|_{C^2(\bar I)} =  \| \hat{H}(\cdot,0) - \lambda(\cdot,\hat{z})\|_{C^2(\bar I)} <\eta,
$$
where the last inequality holds since $\hat{H}$ satisfies \eqref{eq:proph2b}. 
This contradicts the second statement of \eqref{eq:proph2e}.
\end{proof}
\begin{corollary}\label{cor:h2}
For each $\eta_0>0$, there exists $\delta_1>0$ independent of $\hat{z}$ and $\hat{t}$ such that {$V_1(z)$ satisfies the condition \eqref{eq:convexity} of Proposition {\rm\ref{prop:perturbation}}}, 
and if
\begin{equation}\label{eq:corh2.1}
\sup_{\hat t - 3\sqrt\ep \leq t \leq \hat t  - 2\sqrt\ep} \|u_\ep(x,z,t) - V_1(z)\|_{C(\bar D \times \bar I)} \to 0 \quad \text{ as }\ep \to 0,
\end{equation}
then 
\begin{equation}\label{eq:3.6.1}
\liminf_{\ep \to 0} \inf_{I \times [\hat t, \hat t + \delta_1]}\partial^2_{zz}H_\ep(z,t) \geq \delta_1>0.
\end{equation} 
\end{corollary}
\begin{proof}
Apply Proposition \ref{prop:perturbation}, then \eqref{eq:corh2.1} implies
$$
\limsup_{\ep \to 0} \sup_{[\hat t - 2\sqrt\ep, \hat t + 2\delta_1]} \|\int_I (z - \hat{z}) \partial^2_x n_\ep(x,z,t)\,dz\|_{L^2(D)} < C|\delta_1|^{1/3}
$$
for each $\delta_1>0$ small. 
Applying Lemma \ref{lem:Lp}, we again obtain $C'$ independent of $\ep$ and $\delta_1$ such that for each $\delta_1>0$ small, 
$$
\limsup_{\ep \to 0}\sup_{\hat t - \sqrt\ep \leq t \leq \hat t + 2\delta_1} \|\rho_\ep(\cdot,t) - \theta_{\hat{z}}(\cdot)\|_{L^2(D)}  < C'|\delta_1|^{1/3}.
$$
By choosing $\delta_1$ small, 
{ Proposition \ref{prop:h2} says that 
$$
\sup_{[\hat{t}, \hat{t}+\delta_1]}\|\partial^2_{z_1z_1} \lambda(\cdot, \hat z) - \partial^2_{z_1z_1}H_\ep(\cdot, t)\|_{C(\overline{U})} \leq \frac{1}{2} \inf_{I\times I} \lambda_{z_1z_1},
$$
since by {\bf(H1)}
the last term is a fixed positive constant. Hence, by reducing $\delta_1>0$ if necessary,}  
\eqref{eq:3.6.1} holds. 
%
%
%
%
%
\end{proof}

\section{Proof of Theorem \ref{thm:main}}  \label{sec:main}

Define 
$$
T^*:= \sup\{ T >0\, :\, 
\liminf_{\ep \to 0} \inf_{I \times [2\sqrt\ep,T]}\partial^2_{zz}H_\ep(z,t) >0\}.
$$
Our purpose is to prove that $T^*=\infty$, which comes with the desired convergence results.

Applying Proposition \ref{prop:traject2}, the two quantities $\bar{z}_\ep(t)$ and $V_\ep(z,t)$ are well-defined in the time interval $[2\sqrt\ep,T^*]$, for all sufficiently small $\ep$. We will also define for convenience
$$
V_\ep(z,t) = V_0(z),\quad \text{ and }\quad \bar{z}_\ep(t) = \bar{z}_0 \quad \text{ for }z \in I, t \in [0,2\sqrt\ep].
$$


\noindent {\bf Step 1.} In this step, we show that $T^* \geq \delta_1$, where $\delta_1$ is as given in Corollary \ref{cor:h2}.

By constructing super- and sub-solutions of the form
$$
V_0(z) - \frac{1}{2}\ep|\log\ep| \pm C\left\{(t + \ep) 
+\frac{1}{\ep} \left( [z - (b-\sqrt\ep)]^3_+  + [a + \sqrt\ep - z]^3_+ \right)\right\},
$$
it follows by comparison using \eqref{eq:v} that
\begin{equation}\label{eq:require'}
\sup_{0 \leq t \leq 3\sqrt\ep} \|u_\ep(x,z,t) - V_0(z)\|_{C(\bar D \times \bar I)} \to 0.
\end{equation}
Based on \eqref{eq:require'}, one can apply Corollary \ref{cor:h2} to yield $T^* \geq \delta_1$.

\medskip

\noindent {\bf Step 2.}  
In this step, we show that as $\ep\to 0$, 
\begin{equation}\label{eq:step3}
\sup_{D \times I \times [0,T^*]} \left| u_\ep(x,z,t) - V_\ep(z,t)\right| \to 0.
\end{equation}
(Note that the convergence in the time interval $[0,3\sqrt\ep]$ has been done in Step 1.)

Next, we show
\begin{equation}\label{eq:step3.2}
\sup_{D \times I \times [2\sqrt\ep, T^*]} \left| u_\ep(x,z,t) - V_\ep(z,t) -   \int_{2\sqrt\ep}^t H_\ep(\bar{z}_\ep(s),s)~ds\right| \to 0.
\end{equation}
To this end, define
$$
\overline{U}(z,t) = V_\ep(z,t) + \int_{2\sqrt\ep}^t H_\ep(\bar{z}_\ep(s),s)~ds + \mu_\ep + \frac{C}{\ep}\left\{[(a +\sqrt\ep) - z]^3_+ + [z - (b -\sqrt\ep)]^3_+\right\}
$$
and, with $\mu_\ep = \|u_\ep(x,z,2\sqrt\ep) - V_0(z)\|_{C(\overline{D} \times \overline{I})} = o(1)$, 
$$
\underline{U}(z,t) = V_\ep(z,t) + \int_{2\sqrt\ep}^t H_\ep(\bar{z}_\ep(s),s)~ds - \mu_\ep - \frac{C}{\ep}\left\{[(a +\sqrt\ep) - z]^3_+ - [z - (b -\sqrt\ep)]^3_+\right\}
$$
 and \eqref{eq:step3.2} follows by comparison using \eqref{eq:v}.

We now claim that 
\begin{equation}\label{eq:step3.3}
\sup_{2\sqrt\ep \leq t\leq  T^* } \left|\int_{2\sqrt\ep}^t H_\ep(\bar{z}_\ep(s),s)~ds\right| = o(1).
\end{equation}
Indeed, \begin{align*}
\int_D \rho_\ep(x,t)~dx &= \iint_{D \times I} \exp\left(-\frac{u_\ep(x,z,t)}{\ep}\right)dzdx \\ 
&= \frac{1}{\sqrt\ep} \iint_{D \times I} \exp\left( -\frac{V_\ep(z,t)}{\ep} + \frac{\int_{2\sqrt\ep}^tH_\ep(\bar{z}_\ep(s),s)ds + o(1)}{\ep} \right)dzdx.
\end{align*}
Since  $V_\ep(z,t) = -c(t)(z-\bar{z}_\ep(t))^2 + O(|z-\bar{z}_\ep(t)|^3)$, we have
$$
0<\liminf_{\ep\to 0}
\frac{1}{\sqrt{\ep}}
\iint_{D\times I} \exp\left(-\frac{V_\ep(z,t)}{\ep}\right)dzdx
\leq 
\limsup_{\ep\to 0}
\frac{1}{\sqrt{\ep}}
\iint_{D\times I} \exp\left(-\frac{V_\ep(z,t)}{\ep}\right)dzdx
<+\infty
$$
and \eqref{eq:step3.3} follows from this, and $\frac{1}{C} \leq \int_D \rho_\ep(x,t)~dx \leq C$ (from Proposition \ref{prop:apriorisum}).

Combining \eqref{eq:require'}, \eqref{eq:step3.2} and \eqref{eq:step3.3}, we deduce \eqref{eq:step3}.

\medskip

\noindent {\bf Step 3.}  
In this step, we show that
\begin{equation}\label{eq:step4}
\sup_{\sqrt\ep \leq t \leq T^*} \left\| \rho_\ep(\cdot,t) - \theta_{\bar{z}_\ep(t)}\right\|_{C(\bar D)} \to 0 \quad \text{ as }\ep \to 0.
\end{equation}

First, we deduce from \eqref{eq:step3} and Proposition \ref{prop:perturbation} that
$$
\sup_{[\ep,T^*]} \left\| \int_I (z - \bar{z}_\ep(t)) \partial^2_x n_\ep (\cdot,z,t)\,dz\right\|_{C(\bar{D})} \to 0.
$$
Then by taking $t_1 = \ep$ and $t_2 \in [\sqrt\ep,T^*]$ in Lemma \ref{lem:Lp}, we deduce 
\begin{equation}\label{eq:step4.1}
\sup_{\sqrt\ep \leq t \leq T^*} \left\| \rho_\ep(\cdot,t) - \theta_{\bar{z}_\ep(t)}\right\|_{L^2(D)} \to 0.
\end{equation}
The estimate \eqref{eq:step4} follows by interpolating \eqref{eq:step4.1} with the uniform estimate $$
\sup\limits_{\sqrt\ep \leq t \leq T^*} \left\| \rho_\ep(\cdot,t) - \theta_{\bar{z}_\ep(t)}\right\|_{C^\beta(\bar D)} \leq C,
$$ 
which follows from Corollary \ref{cor:apriorisum}(b) and that $\|\bar{z}_\ep(\cdot)\|_{C^{0,1}} \leq C$ (by Proposition \ref{prop:traject2}).

\medskip

\noindent {\bf Step 4.}
By \eqref{eq:step4} and Proposition \ref{prop:h2}, we have
\begin{equation}\label{eq:step5.1}
\sup_{I \times [2\sqrt\ep, T^* - \sqrt\ep]} |H_{\ep}(z,t) - \lambda(z, \bar{z}_\ep(t))| \to 0.
\end{equation}
Since $(V_\ep(z,t),\bar{z}_\ep(t))$ is the unique viscosity solution of \eqref{eq:approxhj}, we may apply the stability theorem for viscosity solution, to deduce that, as $\ep \to 0$, $(V_\ep(z,t),\bar{z}_\ep(t))$ must converge to the unique viscosity solution $(V(z,t),\bar{z}(t))$ of \eqref{eq:limithj}, i.e. 
\begin{equation}\label{eq:step5.2}
\sup_{I \times [2\sqrt\ep, T^*]}| V_\ep(z,t) - {V}(z,t)| \to 0 \quad \text{ and }\quad \sup_{[2\sqrt\ep, T^*]}|\bar{z}_\ep(t))  - \bar{z}(t)| \to 0.
\end{equation}

\medskip
\noindent {\bf Step 5.}
We claim that $T^* = +\infty$. Suppose not, then $T^* < +\infty$. But then the above argument would allow us to extend the approximate trajectory by a fixed time-step $\delta_1>0$. This would yield a contradiction which proves that $T^* = +\infty$.

\medskip

\noindent {\bf Step 6.} {To conclude, it remains to show \eqref{eq:1.4.1} - \eqref{eq:1.4.3}. Indeed, \eqref{eq:1.4.3} follows from $\rho_\ep(x,t) \to  \theta_{\bar{z}_\ep(t)}(x)$ (by \eqref{eq:step4}) and $\bar{z}_\ep(t) \to \bar{z}(t)$ (by \eqref{eq:step5.2}). Next,  \eqref{eq:1.4.2} follows from 
$$
\sup_{D \times I \times [0,T^*]} \left| u_\ep(x,z,t) - V_\ep(z,t)\right| \to 0
$$ 
(by \eqref{eq:step3}) and $V_\ep(z,t) \to V(z,t)$ (by \eqref{eq:step5.2}). Finally, Proposition \ref{prop:traject} says for each $t$, $z\mapsto V(z,t)$ has a unique minimum point at $\bar{z}(t)$, so that
\begin{equation}\label{eq:1.4.1a}
n_\ep(x,z,t) = \exp\left(- \frac{V_\ep(z,t) + o(1)}{\ep} \right) \approx \delta_0(z-\bar{z}(t)) \rho_\ep(x, t),
\end{equation}
where the function $\rho_\ep(x,t)  = \int_I n_\ep(x,z,t)\,dz$ appeared since the above needs to be consistent upon integration over $z \in I$. Combining \eqref{eq:1.4.3} and \eqref{eq:1.4.1a}, we deduce \eqref{eq:1.4.1}.}

\section{Discussions and Generalizations}\label{sec:dis}

In Subsection \ref{sec:6.0}, we discuss a special feature of the effective Hamiltonian $\lambda$ that arises from evolution of random dispersal, as modeled by Laplacian in the spatial variable, which is the fact that $z_1 \mapsto \lambda(z_1,z_2)$ has the same monotonicity of $\alpha$, regardless of choice of $\alpha$ and $z_2$. Subsection \ref{sec:6.1} gives a concrete example of $\alpha(z)$ where {\bf(H1)} can be verified. Subsection \ref{sec:6.2} discusses the generalization to multi-dimensional traits, as motivated by the evolution of conditional dispersal, where the optimal trait is not necessarily the slowest dispersal rate anymore.

\subsection{Monotonicity of the effective Hamiltonian}\label{sec:6.0} 

The effective Hamiltonian can be written as
\begin{equation}\label{eq:Lambda_r}
\lambda(z_1,z_2) = \Lambda(\alpha(z_1),\alpha(z_2)),
\end{equation}
where 
$\Lambda(\alpha_1,\alpha_2)$ is the 
smallest eigenvalue of
\begin{equation}\label{eq:Lambda}
\left\{
\begin{array}{ll}
\alpha_1 \Delta  \Phi + (m(x)  - \Theta_{\alpha_2}(x)) \Phi + \Lambda \Phi = 0 &\text{ for }x \in D,\\
\partial_\nu \Phi = 0 &\text{ for }x \in \partial D,
\end{array}
\right.
\end{equation}
where $\Theta_\alpha(x)$ is the unique positive solution to
\begin{equation}\label{eq:Tambda}
\alpha \Delta \Theta + (m(x) - \Theta)\Theta = 0 \,\,\text{ in }D,\,\,\text{ and }\,\, \partial_\nu \Theta =0 \,\,\text{ on }\partial D.
\end{equation}
Then it is well-known that $\Lambda$ is smooth in $\mathbb{R}_+ \times \mathbb{R}_+$, and that $\partial_{\alpha_1}\Lambda(\alpha_1,\alpha_2) >0$ \cite{Altenberg,DHMP,Nicbms}. We may conclude the following immediately.
$$
{\rm sgn}(\lambda_{z_1}(z_1,z_2)) = \alpha'(z_1) \quad \text{ for all }z_1,z_2.
$$
Next, we observe that when $\alpha(z)$ satisfies {\bf(H1)}, it is necessarily U-shaped and has a unique minimum point, which is where the dominant trait eventually converges. 

\begin{lemma}\label{lem:6.1}
Suppose $\alpha(z)$ is chosen such that {\bf(H1)} holds, then there exists $z_{\min} \in {\rm Int}\, I$ such that  
\begin{equation}\label{eq:6.1a}
\alpha'(z) <0 \quad \text{ in }[a,z_{\min}),\quad \alpha'(z_{\min}) = 0 \quad \text{ and }\quad \alpha'(z) {>0} \quad \text{ in }(z_{\min}, b];
\end{equation}
and that $z_1 \mapsto \lambda(z_1,z_2)$ has the same minimum point $z_{\min}$ for each $z_2 \in [a,b]$, 
\begin{equation}\label{eq:6.1b}
\partial_{z_1}\lambda(z,z_2) <0 \quad \text{ in }[a,z_{\min}),\quad \partial_{z_1}\lambda(z_{\min},z_2) = 0 \quad \text{ and }\quad \partial_{z_1}\lambda(z,z_2){>0} \quad \text{ in }(z_{\min}, b].
\end{equation}
\end{lemma}
\begin{proof}
Differentiating \eqref{eq:Lambda_r} with respect to $z_1$, evaluating at $(z_1,z_2)=(a,a)$, and using  {\bf(H1)}, we have
$$
0 > \partial_{z_1}\lambda(a,a) = \partial_{\alpha_1}\Lambda(\alpha(a),\alpha(a)) \alpha'(a).
$$
Using also that $\partial_{\alpha_1}\Lambda >0$, we conclude that $\alpha'(a) <0$ and similarly $\alpha'(b) >0$. Thus $\alpha$ has at least one interior minimum point. Moreover, for each fixed $z_2$, the mapping 
$$
z_1 \mapsto \lambda(z_1,z_2)= \Lambda(\alpha(z_1),\alpha(z_2))
$$
is convex in $z_1$, it has the desired monotonicity property in $z_1$. This proves the second statement. Still because $\partial_{\alpha_1}\Lambda>0$, we deduce that $\alpha$ also has the desired monotonicity as well. 
\end{proof}

\subsection{An explicit example of $\alpha(z)$}\label{sec:6.1}

We can construct explicitly some U-shaped dispersion $\alpha(z)$, so that the associated $\lambda(z_1,z_2)$ satisfies the convexity assumption {\bf (H1)}. 
For any given interval $I_0 = [\alpha_0,\alpha_0 + L_0]$
with $\alpha_0,\,  L_0>0$, define 
$$
\alpha(z) = \alpha_0 + \frac{1}{k_0} \int_0^z \tan z'\,dz' \quad \text{ for }z \in [-z_M, z_M],
$$
where $k_0:= \sup\limits_{(\alpha_1,\alpha_2) \in I_0 \times I_0} \frac{|\partial^2_{\alpha_1}\Lambda|}{\partial_{\alpha_1}\Lambda}$ and $z_M \in (0,\pi/2)$ is the unique number such that $\int_0^{z_M} \tan z\,dz = k_0L_0$. With this choice of $\alpha$, the associated $\lambda(z_1,z_2) = \Lambda(\alpha(z_1),\alpha(z_2))$ satisfies 
$$
\pm \partial_{z_1} \lambda(\pm z_M,z_2) = \pm \partial_{\alpha_1}\Lambda(\alpha(\pm z_M), \alpha(z_2)) \alpha'(\pm {z_M}) >0,
$$
and, thanks to the definition of $k_0$, we compute
\begin{align*}
\partial^2_{z_1} \lambda(z_1,z_2) &\geq  (\partial_{\alpha_1}\Lambda ) \alpha''(z_1) -  |\partial^2_{\alpha_1}\Lambda| (\alpha'(z_1))^2\\
&= \frac{\partial_{\alpha_1}\Lambda }{k_0} \left[ (\tan z_1)' - \frac{1}{k_0} \frac{ |\partial^2_{\alpha_1}\Lambda|}{\partial_{\alpha_1}\Lambda } \left|\tan z_1\right|^2 \right] \\
&\geq \frac{\partial_{\alpha_1}\Lambda }{k_0} \left[ (\tan z_1)' - \left|\tan z_1\right|^2 \right]  
{\geq}\frac{1}{k_0}\inf_{I_0\times I_0} \partial_{\alpha_1}\Lambda >0.
\end{align*}

\subsection{Evolution of conditional dispersal}\label{sec:6.2}

A limitation in our present study comes from the use of convexity to obtain various regularity results, in particular for the solutions of the constrained Hamilton-Jacobi equation. Here we present an example which motivates to look for more general methods. We have in mind the following model considering evolution of conditional dispersal, see \cite{Hao2017} and the references therein. 
\begin{equation}\label{eq:conditional}
    \epsilon \partial_t n_\ep = \alpha(z) \Delta_x u - \beta(z) \nabla \cdot[ n_\ep \nabla m(x)] + n_\ep(x)(m(x) - \rho_\ep(x,t)) + \ep^2 \partial^2_z n_\ep 
\end{equation}
for $x \in D$, $z\in I$ and $t>0$, with appropriate boundary conditions. Here $\alpha$ is the rate of unconditional dispersal, whereas $\beta$ is the rate of the directed movement up the gradient of the prescribed function $m(x)$. One or both $\alpha,\beta$ can be dependent on the trait variable $z$. 

The effective Hamiltonian is again given by the invasion exponent $\lambda(z_1,z_2)$ and can be similarly defined. 
Assuming $m(x) = x$, $\alpha(z) = z$ and $\beta = q$ for some small constant $q$, it can be shown that 
\eqref{eq:conditional} possesses at least one positive equilibrium solution $\tilde{u}_\ep$, which tends to a Dirac measure supported at two distinct points on the trait interval. In particular, the corresponding effective Hamiltonian is nonconvex. We conjecture that the time dependent problem supports moving Dirac-concentrations supported at two points $(\bar{z}_1(t),\bar{z}_2(t))$, which then converges to their equilibrium position.
\\[5pt]
{\bf Acknowledment.} K.-Y. L. is partially supported by National Science Foundation under grant DMS-1853561.  
K.-Y. L. and B.P. have received funding from the European Research Council (ERC) under the European Union's Horizon 2020 research and innovation programme (grant agreement No 740623). Part of this research is carried out while K.-Y. L. and B.P. are visiting the Institute for Mathematical Sciences at the National University of Singapore, and the 
Institut Henri Poincar\'{e} (IHP). The authors acknowledge support of the  (UAR 839 CNRS-Sorbonne Universit\'{e}), and LabEx CARMIN (ANR-10-LABX-59-01) for the visit at IHP.


\appendix

\section{A priori Estimates}\label{sec:A}

\begin{proposition}\label{prop:apriori}
There exists $\widetilde{C}_1>1$, depending on $V_0(z), m(x), \underline\alpha, \overline\alpha, D$ but independent of~$\ep$, such that
$$
\frac{1}{\widetilde{C}_1} \leq \rho_\ep(x,t) \leq \widetilde{C}_1 \quad \text{ for }~x \in D,~ t \geq 0.
$$
\end{proposition}
\begin{proof}
We prove the proposition in six steps.

\noindent {\bf Step 1.} There exists $C$ such that $\sup_{t \geq 0} \|\rho_\ep(\cdot,t)\|_{L^1(D)} \leq C$.

Integrate \eqref{eq:1.1} over $(x,z) \in D\times I$, and use the Cauchy-Schwarz inequality, we obtain
\begin{equation}\label{eq:L1eq}
\ep \frac{d}{dt} \int_D \rho_\ep(x,t)~dx = \int_D \rho_\ep(x,t)(m(x) - \rho_\ep(x,t))~dx 
\leq   \int_D \rho_\ep~dx \big( m^* - \frac{1}{|D|} \int_D\rho_\ep~dx\big)
\end{equation}
where $m^*=\sup_D m$. Hence we deduce from the differential inequality that
\begin{equation}\label{eq:L1ineq}
\int_D \rho_\ep(x,t)~dx \leq \max\left\{\int_D \rho_\ep(x,0)~dx, m^*|D|\right\}\quad \text{ for all }t \geq 0.
\end{equation}
 It remains to estimate the initial total population $\int_D \rho_\ep(x,0)~dx$ by {\bf (H2)}: 
\begin{align*}
\int_D \rho_\ep(x,0)~dx &= \int_D \int_I \exp(-\frac{u_\ep(x,z,0)}{\ep})~dzdx\\ 
&\leq \frac{C}{\sqrt\ep} \int_I \exp
\left(-\frac{V_0(z)}{\ep}\right)~dz \leq \frac{C}{\sqrt\ep} \int_I \exp\left(-\frac{{K_1}|z-\bar{z}_0|^2}{\ep}\right)~dz \leq C.
\end{align*}
Thus Step 1 is a direct consequence of \eqref{eq:L1ineq}.

\noindent {\bf Step 2.} There exists $C$ such that $\sup\limits_{t \geq \ep}  \|\rho_\ep(\cdot,t)\|_{C(\bar D)} \leq C$.

It suffices 
to show the following assertion: 
\begin{claim}\label{claim:useful}
There exists $C_1$ such that for any $t_1 \geq \ep$, 
$$
\sup_{x \in D}\rho_\ep (x,t_1) \leq C_1 \sup_{t \in [t_1-\ep, t_1]}\int_D\rho_\ep(x,t)~dx.
$$
\end{claim}
To prove the claim, we first extend $n_\ep(x,z,t)$ in the $z$ variable by reflection across $z = b$, and then periodically in $z$ to $D \times \mathbb{R} \times [0,\infty)$. 

Consider, for each $(z_1,t_1) \in I\times [\ep,\infty)$, the rescaled function
\begin{equation}\label{eq:N_ep}
N_\ep(x,y,\tau;z_1,t_1):= n_\ep(x,z_1 + \ep y,t_1 + \ep \tau),
\end{equation}
then $N_\ep$ satisfies, with $z= z_1 + \ep y$, 
\begin{equation}\label{eq:N}
\left\{
\begin{array}{ll}
\partial_t N_\ep - \alpha(z) 
\Delta_x N_\ep - \partial^2_y N_\ep  = N_\ep ( m - \rho_\ep) \leq mN_\ep &\text{ in }  D \times  \mathbb{R}\times [-1,\infty),\\
\partial_\nu N_\ep(x,y,\tau) = 0 &\text{ on }  \partial D \times  \mathbb{R}
\times[-1,\infty).
\end{array}
\right.
\end{equation}
Since $N_\ep$ is a subsolution of a linear, parabolic equation with $L^\infty$ bounded coefficients\footnote{Note that the term $-\rho_\ep$ is dropped.}, we can apply the local maximum principle \cite[Theorem 7.36]{lieberman} (see also \cite[Section 6.2]{chenbook}), to obtain a constant $C_1$ independent of $\ep>0$, $z_1 \in I$  and $t_1 \geq \ep$ such that 
\begin{equation}\label{eq:localmaxprin}
\|N_\ep(x,y,\tau)\|_{L^\infty(D\times(-4/5,4/5)\times (-4/5,0))} \leq C_1 \|N_\ep(x,y,\tau)\|_{L^1(D\times(-1,1) \times (-1,0))}.
\end{equation}
(Note that the spatial domain on both sides of the inequality can be taken to be the same, as a consequence of the Neumann boundary condition across $\partial D$.) 
Next, we write
$$
\rho_\ep(x,t_1) 
= \int_{I} n_\ep(x,z_1,t_1)~dz_1 = \int_{I}N_\ep(x,0,0;z_1,t_1)~dz_1.
$$
Taking supremum in $x \in D$, it follows that
\begin{align*}
\|\rho_\ep(\cdot,t_1)\|_{L^\infty(D)} &\leq  \int_I \|N_\ep(\cdot,0,0;z_1,t_1)\|_{L^\infty(D)}dz_1\\
&\leq  \int_I \|N_\ep(x,y,\tau;z_1,t_1) \|_{L^\infty(D \times(-1/2,1/2)\times (-1/2,0))}~dz_1\\
&\leq C\int_I \|N_\ep(x,y,\tau;z_1,t_1) \|_{L^1( D \times(-1,1) \times (-1,0))}~dz_1\\
&\leq C \int_{-1}^0 \int_D \int_{
(a -\ep, b + \ep)}n_\ep(x,z,t_1 + \ep \tau)~dzdxd\tau\\
&\leq C \int_{-1}^0 \int_D \int_{I}n_\ep(x,z,t_1 + \ep \tau)~dzdxd\tau\\
&\leq C \sup_{t \in [t_1-\ep, t_1]} \int_D \rho_\ep(x,t)~dx,
\end{align*}
where we used the periodicity of $n_\ep$ in the second to last inequalities. This proves Claim \ref{claim:useful}. 

Finally, we take supremum over $t_1 \geq \ep$ on both sides of the conclusion of Claim \ref{claim:useful}, we deduce
$$
\sup_{t \geq \ep} \|\rho_\ep (\cdot,t_1)\|_{L^\infty(D)} \leq C_1 \sup_{t_1 \geq \ep} \left[\sup_{t \in [t_1-\ep, t_1]}\int_D\rho_\ep(x,t)~dx\right] \leq \sup_{t \geq 0} \int_D \rho_\ep(\cdot, t)\,dx \leq C.
$$
This completes Step 2.

\noindent {\bf Step 3.} There exists $C>1$ such that $C^{-1} \leq \rho_\ep(x,t)\leq C$ for $x \in D$ and $t \in [0,\ep]$.

Based on {\bf(H2)} we construct the following lower solution of \eqref{eq:u}:
$$
\underline{U}(z,t):= V_0(z) - \frac{1}{2}\ep|\log\ep| - C_2(t + \ep) 
- \frac{C}{\ep} \left\{[z - (b-\sqrt
{\ep})]^3_+  + [a + \sqrt\ep - z]^3_+ \right\},
$$
where $C$ is  chosen large 
such that 
$\underline{U}_z(b, t)
\le 0\le \underline{U}_z(a, t)$,
and
then 
$C_2$ is chosen\footnote{Since the differential inequality \eqref{eq:u}, after dropping the $-\rho_\ep$ term, is independent of $-\rho_\ep$, the constant $C_2$ can be chosen independent of $\sup_{0 \leq t \leq \ep} \|\rho_\ep(\cdot,t)\|_{C(\bar{D})}$. } large enough so that 
$$
\underline{U}(z,0) \leq u_\ep(x,z,0) \quad \text{ for }x \in D,\, z \in I,
$$ 
and the appropriate differential inequality \eqref{eq:u} is satisfied 
in $ D \times I\times(0,\infty)$.
Hence for $0 \leq t \leq \ep$,
$$
u_\ep(x,z,t) \geq \underline{U}(z,t) \geq \frac{1}{2} V_0(z) - \frac{1}{2}\ep|\log \ep| + O(\ep),
$$
{where the last inequality follows from the fact that $V_0(z)\geq 0$ and is bounded below by a positive constant near  $z=a, b$, so that
$$
\frac{1}{2}V_0(z) - \frac{C}{\ep} \left\{[z - (b-\sqrt
{\ep})]^3_+  + [a + \sqrt\ep - z]^3_+ \right\} \geq 0.
$$}
Integrating in $z$, we find 
\begin{align*}
\rho_\ep(x,t) \leq \frac{C}{\sqrt\ep}\int_I \exp\left(-\frac{V_0(z) + O(\ep)}{\ep}  \right)~dz \leq C \quad \text{ for }x \in D, t \in [0,\ep].
\end{align*}
	This proves the upper bound of Step 3. The lower bound can then be similarly proved, by using the upper bound and considering the upper solution\footnote{Here the constant $C_3$ depends on the quantity $\sup_{0 \leq t \leq \ep} \|\rho_\ep(\cdot,t)\|_{C(\bar{D})}$, which has just been proved to be uniformly bounded.}
$$
\overline{U}(z,t):=V_0(z) - \frac{1}{2}\ep|\log\ep| + C_3(t + \ep)
$$
Notice that, here, the term in curly bracket appearing in the definition of $\underline{U}$ is not needed, as $V_0(z)$ has positive outer derivatives at $z \in \partial I$.
This completes the proof of Step 3.

\noindent {\bf Step 4.} There exists $C_0$ such that $\sup_{t \geq 0}  \|\rho_\ep(\cdot,t)\|_{C(\bar D)} \leq C_0$.

Step 4 is an immediate consequence of Steps 2 and 3.

\noindent {\bf Step 5.}  There exists $C$ such that $\inf_{t \geq 0} \int_D \rho_\ep(x,t)~dx \geq 1/C$.

We may assume that 
$$
\inf_{0 \leq t \leq \ep}\int_D\rho_\ep(x,t)~dx > e^{-C_0} \frac{\inf_D m}{2C_1}
$$
where $C_0$ is given in Step 4, and $C_1$ is given in Claim \ref{claim:useful}. Indeed, by Step 3, such an inequality holds if
we increase $C_0$ when necessary. Next, we assume to the contrary that there exist $t_1 > \ep$ such that 
\begin{equation}\label{eq:apriori.1}
\int_D \rho_\ep(x,t_1)~dx = e^{-C_0} \frac{\inf_D m}{2C_1} \quad \text{ and }\quad \frac{d}{dt}  \int_D \rho_\ep~dx\big|_{t=t_1} \leq 0.
\end{equation} 
By Step 4 and \eqref{eq:L1eq}, the function $A(t) := \int_D \rho_\ep(x,t)~dx$ satisfies the differential inequality $\epsilon\frac{d}{dt}A(t) \geq -C_0 A(t)$, so that for $t \in [t_1-\ep,t_1]$,
$$
A(t) \leq e^{C_0(t_1-t)/\ep} A(t_1) \leq e^{C_0}e^{-C_0} \frac{\inf_D m}{2C_1} = \frac{\inf_D m}{2C_1}.
$$
By Claim \ref{claim:useful}, we deduce that 
$$
\|\rho_\ep(\cdot,t_1)\|_{L^\infty(D)} \leq C_1 \sup_{t \in [t_1-\ep,t_1]} A(t) \leq C_1 \frac{\inf_D m}{2C_1} = \frac{\inf_D m}{2}.
$$
Hence, by \eqref{eq:L1eq},
$$
\ep \frac{d}{dt} \int_D \rho_\ep~dx\big|_{t=t_1} =\int_D (m - \rho_\ep)\rho_\ep~dx  \big|_{t=t_1} \geq \frac{\inf_D m}{2} A(t_1) >0,
$$
which is a contradiction to \eqref{eq:apriori.1}. This proves Step 5.

\noindent {\bf Step 6.} There exists $C$ such that $\rho_\ep(x,t) \geq 1/C$ for $x \in \overline D$ and $t \geq 0$.

For $0 \leq t \leq \ep$, the lower bound is proved in Step 3. For $t \geq \ep$, one first notice that $N_\ep(\tau,x,y;t_1,z_1)$ satisfies a parabolic equation with $L^\infty$ bounded coefficients, and thus satisfies a weak Harnack inequality \cite[Theorem 7.37]{lieberman} (see also \cite[Section 6.5]{chenbook}). One then prove that $\rho_\ep(x,t)$ also satisfies such a weak Harnack inequality, so that for each $t_1 \geq \ep$,
$$
\rho_\ep(x,t_1) \geq C \fint^{t_1-\ep/2}_{t_1-\ep} \int_D \rho_\ep(x,t)~dxdt,
$$
where the last term is bounded from below, as proved in Step 5. This proves Step 6.
Finally, the proposition follows from combining Steps 4 and 6.
\end{proof}

\section{Differentiability of the Principal Bundle}\label{sec:B}

The notion of a normalized principal Floquet bundle (see \cite{Polacik1993exponential}) is a generalization to evolution problems of the notion of principal eigenfunction of an elliptic, or periodic-parabolic operator. Its smooth dependence on parameters is recently established in \cite{Cantrell2021evolution}. 

\subsection{The normalized principal bundle}

Let $D \subset \mathbb{R}^N$ be a smooth bounded domain. Given $\alpha>0$ and $c \in C^{\beta,\beta/2}(\bar D \times \mathbb{R})$, we say that the positive function $\phi_1(x,t)$ is the corresponding principal Floquet bundle if it satisfies
$$
\begin{cases}
\partial_t \phi_1 - z \Delta \phi_1 - c(x,t) \phi_1 = H_1(t)\phi_1 &\text{ for }x\in D,~t\in\mathbb{R},\\
\partial_\nu \phi(x,t) = 0 &\text{ for }x\in \partial D,~t \in \mathbb{R},\\
\phi_1(x,t) >0 &\text{ for }x\in  \bar D,~t \in \mathbb{R}.\\
\end{cases}
$$
The existence and uniqueness is proved in \cite{Miercynzski1997globally}, which is based on the abstract result of \cite{Polacik1993exponential}.

To formulate the smooth dependence on parameters, we need the notion of a normalized principal Floquet bundle. 
\begin{definition}
Given $z >0$ and $c \in C^{\beta,\beta/2}(\bar D \times \mathbb{R})$, we say that the pair
$(\Phi_1(x,t), H_1(t))$
is the corresponding normalized principal Floquet bundle if it satisfies
\begin{equation}\label{eq:bundlen}
\begin{cases}
\partial_t \Phi_1 - z \Delta \Phi_1 - c(x,t) \Phi_1 = H_1(t)\Phi_1 &\text{ for }x\in D,~t\in\mathbb{R},\\
\partial_\nu \Phi_1(x,t) = 0 &\text{ for }x\in \partial D,~t \in \mathbb{R},\\
\int_\Omega \Phi_1(x,t)\,dx \equiv 1 &\text{ for }t\in\mathbb{R},\\
\Phi_1(x,t) >0 &\text{ for }x\in  \bar D,~t \in \mathbb{R}.\\
\end{cases}    
\end{equation}
\end{definition}

\begin{theorem}\label{thm:A1}
For each $z>0$ and $c \in C^{\beta,\beta/2}(\bar D\times \mathbb{R})$, there exists a unique pair  $$(\Phi_1(x,t),H_1(t)) \in C^{2+\beta,1+\beta/2}(\bar D\times \mathbb{R})\times C^{\beta/2}(\mathbb{R})$$ satisfying \eqref{eq:bundlen} in classical sense. 
\end{theorem}
\begin{proof}
The existence and uniqueness of $(\Phi_1,H_1(t))$ follows from the existence of the principal Floquet bundle $\phi(x,t)$, by noting that $H_1(t)$ arises from the normalization $\int_\Omega \Phi_1(x,t)\,dx \equiv 1$; 
See \cite[Theorem A.1]{Cantrell2021evolution} for detail.
\end{proof}
We need the smooth dependence of the normalized principal Floquet bundle, which is recently established proved in \cite{Cantrell2021evolution}.

\begin{proposition}\label{prop:A2}
The normalized principal Floquet bundle, as a mapping 
$$
\begin{array}{rl}
(z,c) & \mapsto (\Phi_1,H_1)\\
\mathbb{R}_+ \times C^{\beta,\beta/2}(\overline\Omega\times \mathbb{R}) & \to C^{2+\beta,1+\beta/2}(\overline\Omega\times \mathbb{R}) \times  C^{\beta/2}(\mathbb{R})
\end{array}
$$
is smooth.
In particular, there exists a constant $C=C(M)$ which is independent of
 $z \in \left[ \frac{1}{M}, M\right]$ and $\|c\|_{C^{\beta,\beta/2}(\bar D \times \mathbb{R})} \leq M$ such that
$$
\|\partial_x \Phi_1\|_\infty + \max_{i=0,1,2,3}\| \partial^i_z \Phi_1\|_\infty 
\leq C,
$$
where $\| \cdot \|$ is the $L^\infty$ norm over $(x,z,t) \in D \times \left[\frac{1}{M},M\right] \times \mathbb{R}$, and
$$
\frac{1}{C} \leq \Phi_1(x,t) \leq C \quad \text{ in }D \times \mathbb{R}.
$$
\end{proposition}
\begin{proof}
The smooth dependence is proved in \cite[Proposition A.4]{Cantrell2021evolution}. It remains to prove the positive upper and lower bounds on $\Phi_1$. For this purpose, we recall the uniform Harnack inequality \cite[Theorem 2.5]{Huska2006}, which says that there exists some positive constant  $C=C(M)$ such that
$$
\sup_{D} \Phi_1(\cdot,t) \leq C \inf_{D} \Phi(\cdot,t) \quad \text{ for }z \in  \left[\frac{1}{M},M\right], \, t \in \mathbb{R}.
$$
Thanks to the normalization $\int_\Omega \Phi_1\,dx \equiv 1$, we obtain
$$
\frac{1}{C|\Omega|} \leq \frac{1}{C}\sup_{D} \Phi_1(\cdot,t) \leq \Phi_1(x,t) \leq C\inf_{D}\Phi_1(\cdot,t) \leq \frac{C}{|\Omega|}.
$$
This completes the proof.
\end{proof}

\section{Uniqueness for the Constrained Hamilton Jacobi Equation}\label{sec:D}

We now establish the uniqueness of solutions to a constrained Hamilton-Jacobi equation in an open, bounded one-dimensional interval $I$ under some monotonicity assumption. We begin with a proposition that does not assume convexity of the Hamiltonian and initial data on the trait variable $x$.


\begin{proposition}\label{propunique}
For $i=1, \, 2$, let $(V_i,\bar{z}_i) \in W^{1,\infty}(I'\times [0,T]) \times BV([0,T])$ be a solution to
\begin{equation}\label{eq:limith}
\left\{
\begin{array}{ll}
\partial_t V + |\partial_z V|^2 + R(z, \bar{z}(t),t)=0 &\text{ for }z \in I', \,t \in [0,T],\\
\partial_z V(z,t) = 0 &\text{ for }z \in \partial I', \,t \in [0,T]\\
\end{array} \right.
\end{equation}
in the viscosity sense, and 
which verifies the initial data and the constraint \begin{equation}\label{eq:limith2}
\left\{
\begin{array}{ll}
V(z,0) = V_0(z) &\text{ for }z \in I',\\
\inf_{z \in I'} V(z,t)  = 0 &\text{ for }t \in [0,T]
\end{array} \right.
\end{equation}
in the classical sense. 
Suppose that $R$ is $\mathcal{C}^2$ in all variables, and 
\begin{equation}\label{eq:barzz}
\bar{z}_i(t) \in {\rm Int}\,I'  \quad \text{ for }\quad i=1,2, \, t \in [0,T], 
\end{equation}
\begin{equation}\label{eq:u2ii}
R(z_1,z_2,t)=0 \quad \text{ if and only if } \quad (z_1,z_2,t) \in \Gamma 
\end{equation}
and 
\begin{equation}\label{eq:u2iii}
\partial_{z_1}R(z,z,t) >0 \,\, \text{ in }\Gamma \quad \text{ or } \quad \partial_{z_1}R(z,z,t) <0 \,\,\text{ in }\Gamma,
\end{equation}
where $\Gamma:= \{(z_1,z_2,t)\in I' \times I' \times [0,T]\,:\, z_1=z_2\}$. Then 
$$
(V_1(z,t),\bar{z}_1(t)) = (V_2(z,t),\bar{z}_2(t))\quad \text{ for }z\in I'\text{ and }t \in [0,T].
$$
\end{proposition}

We postpone the proof of this proposition and conclude the proofs of Propositions~\ref{prop:traject} and~\ref{prop:traject2}.

\subsection{Proofs of Propositions \ref{prop:traject} and~\ref{prop:traject2}}

\begin{proof}[Proof of Proposition \ref{prop:traject}]
Since assertions (ii) and (iii), and the existence part of assertion (i) are proved in \cite{Lorz2011,Mirrahimi2016}, we prove the uniqueness part of assertion (i) in the following.

Let $(V_i(z,t),\bar{z}_i(t))$ be two solutions such that the common initial data $V_0(z)$ is convex, smooth, and attains a unique minimum at some $\bar{z}_0 \in {\rm Int}\, I$.

\noindent {\bf Step 1.} \quad 
Due to the differential equation \eqref{eq:canonical}, and the fact that 
%
$\lambda(z,z)\equiv 0$, exactly one of the following cases holds:
$$
(i)\quad \bar{z}_1(t)\equiv \bar{z}_2(t) \equiv \bar{z}_0,\quad (ii)\quad \left( \frac{d}{dt}\bar{z}_1(t) \right)\left(  \frac{d}{dt}\bar{z}_2(t) \right) >0\quad \text{ for all }t \in [0,T].
$$
In case (i), the conclusion follows from standard uniqueness of viscosity solution to Neumann problem. We henceforth consider case (ii). In fact, by \eqref{eq:canonical} we can assume without loss of generality that 
\begin{equation}\label{eq:pp1}
 \frac{d}{dt} \bar{z}_i(t) >0 \,\,\,\text{ and }\,\,\, \partial_{z_1}\lambda(\bar{z}_i(t),\bar{z}_i(t)) <0 \quad\text{ for }\quad t \in [0,T]\,\text{ and }\, i=1,2.
\end{equation}

\noindent {\bf Step 2.} \quad 
We choose $I'$ to be slightly larger than $[\bar{z}_0, \max_{i=1,2} \bar{z}_i(T)]=\cup_{i=1}^2 \bar{z}_i([0,T])$, then  \eqref{eq:barzz} and \eqref{eq:u2ii} 
hold.

\noindent {\bf Step 3.} \quad 
$(V_1(z,t),\bar{z}_1(t))$ and $(V_2(z,t),\bar{z}_2(t))$ are viscosity solutions to the same constrained Hamilton-Jacobi equation (with Neumann boundary conditions) on the restricted domain $I' \times [0,T]$:
\begin{equation}\label{eq:limith'}
\left\{
\begin{array}{ll}
\partial_t V + |\partial_z V|^2 - \lambda(z, \bar{z}(t))=0 &\text{ for }z \in I', \,t \in [0,T],\\
\partial_z V(z,t) = 0 &\text{ for }z \in \partial I', \,t \in [0,T],\\
V(z,0) = V_0(z) &\text{ for }z \in I',\\
\inf_{z \in I'} V(z,t)  = 0 &\text{ for }t \in [0,T].
\end{array} \right.
\end{equation}
(The initial data and constraint are satisfied in the classical sense.)
This step is valid since $z\mapsto V_i(z,t)$ is convex, and the unique minimum point $\bar{z}_i(t) \in {\rm Int}\,I'$ for all $t$.

\noindent {\bf Step 4.} \quad 
We can now apply Propositon \ref{propunique} to conclude that $\bar{z}_1(t)=\bar{z}_2(t)$ a.e. in $ [0,T]$. Then we can use the variational characterization to deduce that $V_1\equiv V_2$ in the original domain $I\times [0,T]$ (not just in the smaller domain  $I'\times [0,T]$).
\end{proof}

\begin{proof}[Proof of Proposition \ref{prop:traject2}]
Fix $\epsilon>0$ and let $R(z_1,z_2,t):= -H_\ep(z_1,t)+H_\ep(z_2,t)$. Once again, the existence of a viscosity solution $(V_\ep(z,t),\bar{z}_\ep(t))$ holds. Since $V_0(z)$ and $-R(z,\bar{z},t)$ are convex in $z$, it follows that $V_\ep(z,t)$ is strictly convex in $z$ for each $t$, and satisfies the differential equation
\begin{equation}\label{eq:DE}
\frac{d}{dt} \bar{z}_\ep(t) =\frac{1}{\partial_{zz}V_\ep( \bar{z}_\ep(t),t)} \partial_{z_1}R( \bar{z}_\ep(t), \bar{z}_\ep(t),t).
\end{equation}
To show uniqueness, it suffices to repeat the proof of Proposition \ref{prop:traject}. We omit the details.
\end{proof}

\subsection{Proof of Proposition \ref{propunique}}
Suppose two sets of solutions $(V_i,\bar{z}_i)$, $i=1,2,$ are given. 
First, extend the problem by reflection to the domain $[2\inf I' - \sup I', \sup I']\times [0,T]$ and then extend it periodically so that it is defined in $\mathbb{R} \times [0,T]$. We use the variational characterization:
\begin{equation}\label{eq:variationp}
V_i(z,t) = \inf_{\gamma(t) = z} \left\{ \int_0^t \left[\frac{|\dot\gamma(s)|^2}{4} - R(\gamma(s), \bar{z}_i(s),s)\right]\,ds + V_0(\gamma(0)) \right\}
\end{equation}
with the understanding that $V_0(z)$ is also being extended evenly and periodically so that it is defined for all $z \in \mathbb{R}$.

\begin{lemma}\label{lem:mono1p}
Assume \eqref{eq:barzz}, \eqref{eq:u2ii} and the first alternative of \eqref{eq:u2iii}.
Let $(V(z,t),\bar{z}(t)) \in W^{1,\infty}(\mathbb{R} \times [0,T]) \times \textup{BV}([0,T])$ be a solution of \eqref{eq:limith} and \eqref{eq:limith2}. Then, at possible discontinuity points we have
\begin{equation}\label{eq:mmp}
\bar{z}(t-)\leq \bar{z}(t+)\quad \text{ and }\quad \mathcal{N}(t) \subset [\bar{z}(t-), \bar{z}(t+)]
\end{equation}
where $\mathcal{N}(t) = \{z\in \bar I': V(z,t) = 0\}$. 
\end{lemma}
\begin{proof}
By \eqref{eq:u2ii} and the first alternative of \eqref{eq:u2iii}, we have
\begin{equation}\label{eq:u2iv}
{\rm sgn}(R(z_1,z_2,t)) = {\rm sgn}(z_1-z_2).
\end{equation}
Fix $(z,t)$ such that $V(z,t) = 0$, it suffices to show that $\bar{z}(t-) \leq z \leq \bar{z}(t+)$. To show the first inequality, choose a minizing curve $\gamma(t)$ for \eqref{eq:variationp} such that $\gamma(t) = z$ and  
\begin{equation}\label{eq:ztp}
0=V(z,t) = \int_0^t \frac{|\dot\gamma(s)|^2}{4} - R(\gamma(s), \bar{z}(s),s)\,ds + V_0(\gamma(0)).
\end{equation}
In fact, for any $h \in (0,t)$, the dynamic programming principle says that
$$
0 \leq V(\gamma(t-h),t-h) = \int_0^{t-h} \left[\frac{|\dot\gamma(s)|^2}{4} -R(\gamma(s),
\bar{z}(s),s)\right]\,ds + V_0(\gamma(0)) .
$$
Subtracting, we have
$$
0 \leq -\int_{t-h}^t\left[ \frac{|\dot\gamma(s)|^2}{4} -R(\gamma(s), \bar{z}(s),s)\right]\,ds \leq  \int_{t-h}^t R(\gamma(s),\bar{z}(s),s)\,ds \quad \text{ for all }s \in (0,t).
$$
Dividing by $h$ and letting $h \searrow 0$, we deduce that 
\begin{equation}\label{eq:firstineqp}
R(z,\bar{z}(t-),t) = R(\gamma(t),\bar{z}(t-),t) \geq 0 .  
\end{equation}
By \eqref{eq:u2iv}, we have $\bar{z}(t-) \leq z.$

Next, we fix as above $z,t,\gamma(\cdot)$, and define $\gamma_1:[0,t+1] \to \mathbb{R}$ by
$$
\gamma_1(s) = \left\{
\begin{array}{ll}
\gamma(s) & \text{ for } 0 \leq s \leq t,\\
z &\text{ for }s > t.
\end{array}
\right.
$$
Then by \eqref{eq:variationp}, we have for $0< h < 1$,
$$
0 \leq V(t+h, \gamma_1(t+h)) \leq \int_0^{t+h} \frac{|\dot\gamma_1(s)|^2}{4} -R(\gamma_1(s), \bar{z}(s),s)\,ds + V_0(\gamma(0)).
$$
Subtracting \eqref{eq:ztp} from the above, we have
\begin{equation}\label{eq:secondineqp}
0 \leq -\int_t^{t+h} R(z, \bar{z}(s),s)\,ds .
\end{equation}
Divide by $h$, and let $h \to 0$, then $R(z, \bar{z}(t+),s) \leq 0$. By \eqref{eq:u2iv}, we have $\bar{z}(t+) \geq z$.
%
\end{proof}
\begin{remark}\label{rmk:B2p}
From the above proof, for each $t_0>0$ and $\bar{z}_0 \in \mathcal{N}(t_0)$, it follows from 
\eqref{eq:secondineqp} that
$$
\int_{t_0}^{t_0 + h}R(\bar{z}_0, \bar{z}(s),s)\,ds \leq 0 \quad \text{ for all sufficiently small }h>0.$$
\end{remark}
\begin{lemma}\label{lem:mono2p}
Assume \eqref{eq:barzz}, \eqref{eq:u2ii} and the first alternative of \eqref{eq:u2iii}.
Let $(V(z,t),\bar{z}(t)) \in W^{1,\infty}(\mathbb{R} \times [0,T]) \times \textup{BV}([0,T])$ be a solution of \eqref{eq:limith} and \eqref{eq:limith2}, then $\bar{z}$ is non-decreasing. Furthermore, we have 
\begin{equation}\label{eq:pp3}
\lim_{t \to 0+} \left[ \sup\mathcal{N}(t)\right] = \sup \mathcal{N}(0),
\end{equation}
and, still with $\mathcal{N}(t) = \{z\in \bar I': V(z,t) = 0\}$, we have
\begin{equation}\label{eq:pp2}
\bar{z}(t+) = \sup \mathcal{N}(t).
\end{equation}
\end{lemma}
\begin{proof}
Again, 
without loss of generality, we choose the right-continuous representative of $\bar{z}$.  Now, assume to the contrary that $\bar{z}$ is not non-decreasing. i.e. $\bar{z}(t_1) > \bar{z}(t_2)$ for some $t_1 < t_2$. Since $\bar{z}$ is right-continuous, we have $\bar{z}(t_1) > \bar{z}(t_2+)$, i.e. there exists $t_3 > t_2$ such that $\bar{z}(t_1) > \bar{z}(t)$ for all $t \in [t_2,t_3]$. Let $t_0 = \sup\{ t \in [t_1,t_3)\,:\, \bar{z}(t) \geq \bar{z}(t_1)\}$. Then $t_0 \leq t_2 < t_3$, and
\begin{equation}\label{eq:4.2p}
\bar{z}(t) < \bar{z}(t_1) \leq \bar{z}(t_0-) \quad \text{ for } t \in (t_0,t_3).
\end{equation}
In particular, we deduce that $\bar{z}(t_0+) \leq \bar{z}(t_0-)$. By \eqref{eq:mmp}, it follows that $\bar{z}(t_0+) = \bar{z}(t_0-)$, and that $\mathcal{N}(t_0)$ consists of a single element, which we denote by $\bar{z}_0$.

By the first alternative of \eqref{eq:u2iii}, there exists a small $\delta_1>0$ such that
\begin{equation}\label{eq:thirdineqp}
R(\bar{z}_0,z,t)  >0 \quad \text{ for }z \in (\bar{z}_0 - \delta_1, \bar{z}_0),\, t \in [t_0-\delta_1,t_0].
\end{equation}
Then choose $h\in (0,\delta_1)$ small enough so that $\bar{z}(s) \in (\bar{z}_0 - \delta_1,  \bar{z}_0)$ for all $s \in (t_0,t_0+h)$ (which is guaranteed by $\bar{z}(t_0+) =\bar{z}(t_0-)= \bar{z}_0$ and \eqref{eq:4.2p}). We deduce from \eqref{eq:thirdineqp} that
$$
\int_{t_0}^{t_0+h} R(\bar{z}_0, \bar{z}(s),s)\,ds >0,
$$
where the last inequality follows from $\bar{z}(s)<\bar{z}_0$ (by \eqref{eq:4.2p}). We obtain a contradiction to Remark \ref{rmk:B2p}. 

Next, we observe that \eqref{eq:pp3} follows directly from the upper-semicontinuity of $\mathcal{N}$, i.e. 
\begin{equation}\label{eq:pp4}
\limsup\limits_{t' \to t+}  \mathcal{N}(t') \subset \mathcal{N}(t).
\end{equation}

It remains to show \eqref{eq:pp2}.  It follows from \eqref{eq:mmp} (proved in Lemma \ref{lem:mono1p}) that $\bar{z}(t+) \geq \sup \mathcal{N}(t)$. Moreover, \eqref{eq:mmp} and the monotonicity of $\bar{z}$ imply that  
$$
\bar{z}(t+) \leq \bar{z}(t'-)\leq \inf\mathcal{N}(t') \leq \sup \mathcal{N}(t') \quad \text{ for each }t' >t. 
$$
Using \eqref{eq:pp4}, we may let $t' \to t+$ to deduce $\bar{z}(t+) \leq  \sup\mathcal{N}(t)$. This proves \eqref{eq:pp2}.
\end{proof}

\begin{proof}[Proof of Proposition \ref{propunique}]
For $i=1,2$, let $(V_i(z,t),\bar{z}_i(t)) \in W^{1,\infty}(\mathbb{R} \times [0,T])\times BV[0,T]$ be two solutions of \eqref{eq:limith} and \eqref{eq:limith2}. It suffices to show that $\bar{z}_1(t) = \bar{z}_2(t)$ a.e. in $[0,T]$. 
 Without loss of generality, one can reduce to the case that for each $t>0$,  $\bar{z}_1 \neq \bar{z}_2$ in a set of positive measure in $(0,t)$.
Furthermore, by considering $z' = -z$ if necessary, we can assume the first alternative of \eqref{eq:u2iii} to hold.

To apply \cite[Section 3, Remark 3]{Calvez2018}, it remains to verify {\bf(U1)-(U3)}. Now,  observe that the Hamiltonian function in \eqref{eq:limith} is smooth and satisfies $H(\bar z,t,z,p) = |p|^2 +R(z, \bar z,t)$, so that $L(\bar z, z, t,v) = |v|^2/4 - R(z, \bar z,t)$. This verifies {\bf(U1)}. The condition {\bf(U2)} also holds,  as $V_i(z,t)$ admits the variational characterization \eqref{eq:variationp}.

It suffices to check {\bf(U3)}. Define 
$$
z^t_i = \sup \mathcal{N}_i(t),\quad \text{ where }\quad \mathcal{N}_i(t) = \{z \in I: V_i(z,t) = 0\},$$ 
and let $\gamma^t_i \in AC[0,t]$ be the minimizing path corresponding to the value $V_i(z^t_i,t) = 0$. We need to verify
the following three conditions:
\begin{itemize}
\item[{\rm(i)}] $\lim\limits_{t \to 0+} z^t_i = \bar{z}_0:= \sup\{z \in I: V_0(z)=0\}$,
\item[{\rm(ii)}] $\displaystyle \inf\limits_{0 < \theta < 1} \partial_{z_2}R(\lim_{t \to 0+}[(1-\theta)\bar{z}_1(t) + \theta \bar{z}_2(t)], \bar{z}_0,0) <0$,
\item[{\rm(iii)}] $\limsup\limits_{t\to 0+}\|\dot\gamma^t_i\|_{L^\infty(0,t)} <+\infty$.
\end{itemize}

To verify condition (i), we apply \eqref{eq:pp3} and \eqref{eq:pp2} to get
\begin{equation}\label{eq:pp5}
\lim_{t \to 0+}\bar{z}_1(t) = \lim_{t \to 0+}\bar{z}_2(t) = \lim_{t \to 0+} z^t_i =\bar{z}_0. 
\end{equation}

To verify condition (ii), observe that \eqref{eq:u2ii} and the first alternative of \eqref{eq:u2iii} imply that $\partial_{z_2}R(\bar{z}_0,\bar{z}_0,0) = -\partial_{z_1}R(\bar{z}_0,\bar{z}_0,0) <0$. Using \eqref{eq:pp5}, we can then compute 
$$ \sup_{0 < \theta < 1} \partial_{z_2} R(\lim_{t \to 0+}[(1-\theta)\bar{z}_1(t) + \theta \bar{z}_2(t)], \bar{z}_0,0) = \partial_{z_2}R(\bar{z}_0,\bar{z}_0,0)  <0.$$

For condition (iii), we observe that the initial data $g$ and the Lagrangian function $\frac{|v|^2}{4} - R(I,t,x)$ are periodic in $x$, so the minimizing paths $\gamma^t_i$ (corresponding to $(z^t_i,t)$) exists and is uniformly bounded in $L^\infty$. One can derive the regularity of $\dot\gamma^t_i$ by repeating the arguments in \cite[Section2]{Calvez2018}.

Having verified {\bf(U1)-(U3)}, one can then invoke  \cite[Section 3, Remark 3]{Calvez2018} to yield a contradiction. This proves that $\bar{z}_1(t) \equiv \bar{z}_2(t)$ a.e. in $[0,T]$. That $V_1(z,t) \equiv V_2(z,t)$ follows from the standard uniqueness results.
\end{proof}

\bibliographystyle{siam} 
\bibliography{refs} 

\begin{thebibliography}{10}

\bibitem{Altenberg}
{\sc L.~Altenberg}, {\em Resolvent positive linear operators exhibit the
  reduction phenomenon}, Proc. Natl. Acad. Sci. USA, 109 (2012),
  pp.~3705--3710.

\bibitem{Barles2013introduction}
{\sc G.~Barles}, {\em An introduction to the theory of viscosity solutions for
  first-order {H}amilton-{J}acobi equations and applications}, in
  Hamilton-{J}acobi equations: approximations, numerical analysis and
  applications, vol.~2074 of Lecture Notes in Math., Springer, Heidelberg,
  2013, pp.~49--109.

\bibitem{Barles2009}
{\sc G.~Barles, S.~Mirrahimi, and B.~Perthame}, {\em Concentration in
  {L}otka-{V}olterra parabolic or integral equations: a general convergence
  result}, Methods Appl. Anal., 16 (2009), pp.~321--340.

\bibitem{BouinH17}
{\sc E.~Bouin and C.~Henderson}, {\em Super-linear spreading in local bistable
  cane toads equations}, Nonlinearity, 30 (2017), pp.~1356--1375.

\bibitem{BouinHR17b}
{\sc E.~Bouin, C.~Henderson, and L.~Ryzhik}, {\em The {B}ramson logarithmic
  delay in the cane toads equations}, Quart. Appl. Math., 75 (2017),
  pp.~599--634.

\bibitem{BouinHR17a}
\leavevmode\vrule height 2pt depth -1.6pt width 23pt, {\em Super-linear
  spreading in local and non-local cane toads equations}, J. Math. Pures Appl.
  (9), 108 (2017), pp.~724--750.

\bibitem{Bouin2015}
{\sc E.~Bouin and S.~Mirrahimi}, {\em A {H}amilton-{J}acobi approach for a
  model of population structured by space and trait}, Commun. Math. Sci., 13
  (2015), pp.~1431--1452.

\bibitem{Calvez2018}
{\sc V.~Calvez and K.-Y. Lam}, {\em Uniqueness of the viscosity solution of a
  constrained {H}amilton-{J}acobi equation}, Calc. Var. Partial Differential
  Equations, 59 (2020), pp.~Paper No. 163, 22.

\bibitem{Cantrell2021evolution}
{\sc R.~S. Cantrell and K.-Y. Lam}, {\em On the evolution of slow dispersal in
  multispecies communities}, SIAM J. Math. Anal., 53 (2021), pp.~4933--4964.

\bibitem{Champagnat2006}
{\sc N.~Champagnat, R.~Ferri\`{e}re, and S.~M\'{e}l\'{e}ard}, {\em Unifying
  evolutionary dynamics: From individual stochastic processes to macroscopic
  models}, Theoretical Population Biology, 69 (2006), pp.~297--321.
\newblock ESS Theory Now.

\bibitem{chenbook}
{\sc Y.-Z. Chen}, {\em Second Order Parabolic Equations (in Chinese), Beijing
  University Mathematics series}, Beijing University Press, Beijing, 2003.

\bibitem{Desvillettes2008}
{\sc L.~Desvillettes, P.-E. Jabin, S.~Mischler, and G.~Raoul}, {\em On
  selection dynamics for continuous structured populations}, Commun. Math.
  Sci., 6 (2008), pp.~729--747.

\bibitem{Dieckmann1996}
{\sc U.~Dieckmann and R.~Law}, {\em The dynamical theory of coevolution: a
  derivation from stochastic ecological processes}, J. Math. Biol., 34 (1996),
  pp.~579--612.

\bibitem{odo}
{\sc O.~Diekmann}, {\em A beginner's guide to adaptive dynamics}, in
  Mathematical modelling of population dynamics, vol.~63 of Banach Center
  Publ., Polish Acad. Sci. Inst. Math., Warsaw, 2004, pp.~47--86.

\bibitem{Diekmann2005}
{\sc O.~Diekmann, P.-E. Jabin, S.~Mischler, and B.~Perthame}, {\em The dynamics
  of adaptation: An illuminating example and a hamilton–jacobi approach},
  Theoretical Population Biology, 67 (2005), pp.~257--271.

\bibitem{DHMP}
{\sc J.~Dockery, V.~Hutson, K.~Mischaikow, and M.~Pernarowski}, {\em The
  evolution of slow dispersal rates: a reaction diffusion model}, J. Math.
  Biol., 37 (1998), pp.~61--83.

\bibitem{Evans1989}
{\sc L.~C. Evans}, {\em The perturbed test function method for viscosity
  solutions of nonlinear {PDE}}, Proc. Roy. Soc. Edinburgh Sect. A, 111 (1989),
  pp.~359--375.

\bibitem{Geritz2008}
{\sc S.~Geritz and M.~Gyllenberg}, {\em The mathematical theory of adaptive
  dynamics}, Cambridge UniversityPress, Cambridge,  (2008).

\bibitem{Hao2017}
{\sc W.~Hao, K.-Y. Lam, and Y.~Lou}, {\em Concentration phenomena in an
  integro-{PDE} model for evolution of conditional dispersal}, Indiana Univ.
  Math. J., 68 (2019), pp.~881--923.

\bibitem{Hastings1983}
{\sc A.~Hastings}, {\em Can spatial variation alone lead to selection for
  dispersal?}, Theoretical Population Biology, 24 (1983), pp.~244--251.

\bibitem{HendersonPS2018}
{\sc C.~Henderson, B.~Perthame, and P.~E. Souganidis}, {\em Super-linear
  propagation for a general, local cane toads model}, Interfaces Free Bound.,
  20 (2018), pp.~483--509.

\bibitem{Huska2006}
{\sc J.~H\'{u}ska}, {\em Harnack inequality and exponential separation for
  oblique derivative problems on {L}ipschitz domains}, J. Differential
  Equations, 226 (2006), pp.~541--557.

\bibitem{Jabin2011}
{\sc P.-E. Jabin and G.~Raoul}, {\em On selection dynamics for competitive
  interactions}, J. Math. Biol., 63 (2011), pp.~493--517.

\bibitem{Jabin2016}
{\sc P.-E. Jabin and R.~S. Schram}, {\em Selection-mutation dynamics with
  spatial dependence}, 2016, preprint.

\bibitem{Lam2017b}
{\sc K.-Y. Lam}, {\em Stability of {D}irac concentrations in an integro-{PDE}
  model for evolution of dispersal}, Calc. Var. Partial Differential Equations,
  56 (2017), pp.~Paper No. 79, 32.

\bibitem{Lam2017a}
{\sc K.-Y. Lam and Y.~Lou}, {\em An integro-{PDE} model for evolution of random
  dispersal}, J. Funct. Anal., 272 (2017), pp.~1755--1790.

\bibitem{lieberman}
{\sc G.~M. Lieberman}, {\em Second order parabolic differential equations},
  World Scientific Publishing Co., Inc., River Edge, NJ, 1996.

\bibitem{Lorz2011}
{\sc A.~Lorz, S.~Mirrahimi, and B.~Perthame}, {\em Dirac mass dynamics in
  multidimensional nonlocal parabolic equations}, Comm. Partial Differential
  Equations, 36 (2011), pp.~1071--1098.

\bibitem{Miercynzski1997globally}
{\sc J.~Mierczy\'{n}ski}, {\em Globally positive solutions of linear parabolic
  {PDE}s of second order with {R}obin boundary conditions}, J. Math. Anal.
  Appl., 209 (1997), pp.~47--59.

\bibitem{Mirrahimi2015}
{\sc S.~Mirrahimi and B.~Perthame}, {\em Asymptotic analysis of a selection
  model with space}, J. Math. Pures Appl. (9), 104 (2015), pp.~1108--1118.

\bibitem{Mirrahimi2016}
{\sc S.~Mirrahimi and J.-M. Roquejoffre}, {\em A class of {H}amilton-{J}acobi
  equations with constraint: uniqueness and constructive approach}, J.
  Differential Equations, 260 (2016), pp.~4717--4738.

\bibitem{Nicbms}
{\sc W.-M. Ni}, {\em The Mathematics of Diffusion}, vol.~82 of CBMS-NSF
  Regional Conference Series in Applied Mathematics, Society for Industrial and
  Applied Mathematics (SIAM), Philadelphia, PA, 2011.

\bibitem{Nordmann2020}
{\sc S.~Nordmann and B.~Perthame}, {\em Dynamics of concentration in a
  population structured by age and a phenotypic trait with mutations.
  {C}onvergence of the corrector}, J. Differential Equations, 290 (2021),
  pp.~223--261.

\bibitem{Perthame2008}
{\sc B.~Perthame and G.~Barles}, {\em Dirac concentrations in
  {L}otka-{V}olterra parabolic {PDE}s}, Indiana Univ. Math. J., 57 (2008),
  pp.~3275--3301.

\bibitem{Perthame2015}
{\sc B.~Perthame and P.~E. Souganidis}, {\em Rare mutations limit of a steady
  state dispersal evolution model}, Math. Model. Nat. Phenom., 11 (2016),
  pp.~154--166.

\bibitem{Polacik1993exponential}
{\sc P.~Pol\'{a}\v{c}ik and I.~Tere\v{s}\v{c}\'{a}k}, {\em Exponential
  separation and invariant bundles for maps in ordered {B}anach spaces with
  applications to parabolic equations}, J. Dynam. Differential Equations, 5
  (1993), pp.~279--303.

\end{thebibliography}

\end{document}